\documentclass[12pt]{article}
{\begin{Sbox}\begin{minipage}}%
{\end{minipage}\end{Sbox}\fbox{\TheSbox}}%

\usepackage[psamsfonts]{amssymb}
\usepackage{amsmath, amsthm, anysize, enumerate, url}
\usepackage{color}
\usepackage{graphicx}
\usepackage{bm}
\usepackage{subfigure}
\usepackage[round]{natbib}
\usepackage{slashed}

\usepackage{mathrsfs}
\usepackage{float}
\newtheorem{theorem}{Theorem}[section] 

\newtheorem{corollary}{Corollary}[section]

\newtheorem{lemma}{Lemma} [section] 
\newtheorem{proposition}{Proposition}[section] 
\newtheorem{assumption}{Assumption}[section]
\newtheorem{proof of lemma}{Proof of Lemma}
\newtheorem{proof of theorem}{Proof of Theorem}
\usepackage[colorlinks,
                   linkcolor=red,
                  anchorcolor=blue,
                  citecolor=blue
                  ]{hyperref}

\numberwithin{equation}{section}
\newtheorem{definition}{Definition}[section]

\usepackage{mathrsfs}

\theoremstyle{definition}
\newtheorem{example}{Example}[section]

\newcommand{\E}{\mathrm{E}}
\newcommand{\pr}{\mathrm{P}}
\newcommand{\R}{\mathbb{R}}

\newcommand{\var}{\operatorname{var}}

\newcommand{\bs}{\boldsymbol}
\usepackage{geometry}
\usepackage{setspace}

\numberwithin{equation}{section}
\theoremstyle{plain}

\usepackage{indentfirst}
\usepackage{natbib}
\begin{document}
\vspace*{0.15cm}
\begin{center}
{\Large\bf Asymptotic in a class of network models with an increasing sub-Gamma degree sequence}
\vskip 1.2\baselineskip
{\large Jing Luo$^{1,2}$,  \ Haoyu Wei$^{3*}$, \ Xiaoyu Lei$^{4}$, \ Jiaxin Guo$^{5}$}
\vskip 0.3\baselineskip
{\sl ${1}$. Department of Statistics, South-Central Minzu University, Wuhan, China.}
\vskip 0.2\baselineskip
{\sl ${2}$. School of Mathematics and Statistics, and Key Laboratory of Nonlinear Analysis and Applications (Ministry of Education), Central China Normal University, Wuhan, China.}
\vskip 0.2\baselineskip
{\sl ${3}$.  $^{*}$Department of Economics, University of California San Diego, La Jolla, USA.}
\vskip 0.3\baselineskip
{\sl ${4}$. Department of Statistics, University of Wisconsin–Madison, Madison, USA}
\vskip 0.2\baselineskip
{\sl ${5}$. University of Warwick, Coventry, CV4 7AL, UK.}

\vskip 0.1\baselineskip

\end{center}
\footnotetext[3]{$^*$Correspondence author. Email: h8wei@ucsd.edu (Haoyu Wei). Luo's research is partially supported by the Fundamental Research Funds for the Central Universities of South-Central Minzu University(Grant Number:CZQ22003) and by National Statistical Science Research of China (2022LY051) and by the Open Research Fund of Key Laboratory of Nonlinear Analysis \& Applications (Central China Normal University), Ministry of Education, P.R. China. \\
Jing Luo and Haoyu Wei are co-first authors. Email:jingluo2017@mail.scuec.edu.cn (Jing Luo), {xlei35@wisc.edu} (Xiaoyu Lei), phd19jg@mail.wbs.ac.uk (Jiaxin Guo).}
\vskip 1.5mm

\begin{abstract}
For the differential privacy under the sub-Gamma noise, we derive the asymptotic properties of a class of network models with binary values with a general link function. In this paper, we release the degree sequences of the binary networks under a general noisy mechanism with the discrete Laplace mechanism as a special case. We establish the asymptotic result including both consistency and asymptotically normality of the parameter estimator when the number of parameters goes to infinity in a class of network models. Simulations and a real data example are provided to illustrate asymptotic results.


\end{abstract}

\vskip 5 pt
%

\section{Introduction}
In network data analysis, the privacy of network data has received wide attention as the disclosed data is likely to contain sensitive information about individuals and their social relationships (sexual relationships, email exchanges, money transfers, etc). Analyzing this type of data can uncover valuable information that can be used to address many important social concerns, such as disease transmission, fraud detection, precision marketing, and more. In recent years, the urgent need to solve the problem of network privacy protection has led to the rapid development of algorithms for securely publishing network data or aggregating network data (see \cite{Zhou2008A}, \cite{Yuan2011Personalized}, \cite{Cutillo2010Privacy}, and \cite{Lu2014Exponential}). However, the unstructured characteristics of network data bring great challenges to the statistical inference (see \cite{fienberg2012brief}, \cite{Mosler2015}) . Data privacy protection is usually achieved by adding some noise to the data. The most typical method is differential privacy[\cite{Dwork2006}]. Due to the unstructured characteristics of network data and the data with noise, there are relatively few theoretical studies on statistical asymptotic behavior of noise-based network data.

The Erd\"{o}s-R\'{e}nyi model (see \cite{Erd1959On}) is generally acknowledged as one of the earliest
random binary graph models, in which each edge occurs with the same probability independent
of any other edge. However, it lacks the ability
to capture the extent of degree heterogeneity commonly associated with network data in
practice. To capture the heterozygous of nodes' degree, a class
of models have come into use for analyzing binary undirected networks network data.
The simplest network model is the binary network model, which mainly uses the degree sequence to capture the real network [\cite{albert2002}]. The $\beta$-model has undirected binary weighted, which is known for binary arrays whose distribution only depends on the row and column totals (see \cite{degree1britton2006generating,degree3bickel2011method,degree4zhao2012consistency,hillar2013maximum}). 
When the number of network nodes tends to infinity, many literature have studied the Maximum Likelihood Estimator (MLE) of the $\beta$-model(see \cite{chatterjee2011random}; \cite{degree2blitzstein2011sequential}; \cite{rinaldo2013maximum}). The asymptotic normality of the MLE of the $\beta$-model is further studied by \cite{yanxu2013}. \cite{Fan2002Connected} proposes another type of binary network model. This model is called log-linear model, where the edge probability
$p_{ij}$ between vertices $i$ and $j$ is $w_{i}w_{j}/(\sum_{k=1}^{n} w_{k})$ under the normalization constraint $w_{i}^{2}\leq (\sum_{k=1}^{n} w_{k}) $($i=1,...,n$), where $w_{i}$ is referred to as the weight of vertex $i$.
 Moreover, Olhede and Wolfe (2012) obtained the approximate value of the parameter estimation of a class of binary network in the case of limited sample network nodes. Until recent years,  \cite{Karwa2016} has given research on the asymptotic theory of adding discrete Laplace noise to network data.  However, when the general noise is added to a class of this network model, the asymptotic theory of parameter estimators is still unknown.

In this paper, we mainly study the asymptotic theory of a class of network models with sub-Gamma degree sequences. It is different from the method of adding noise to network data in \cite{Karwa2016},  \cite{luo2022asymptotic} and \cite{luo2021Ordered}. Therefore, they only considered the asymptotic theory of the parameters for the $\beta-$model in the case of Laplace noise. 
Here, we consider the general distribution of noise variables, and Laplace distribution is only a special case.
And then, our method does not need to deal with the degree sequence of noise, so we can directly use the degree sequence with noise to study the statistical inference of the network model. Furthermore,
the probability-mass or density function of the edge $a_{ij}$ only depends on the sum of $\alpha_i^{*}$ and $\alpha_j^{*}$, where $\alpha_i^{*}$ denotes the strength parameter of vertex $i$. These are the main contributions of this paper.

For the rest of this article, we state as follows. In Section 2, we give null models for undirected network data and sub-Gamma degree sequences. In Section 3, we give a uniform asymptotic result. In Section 4, we illustrate several applications of our main results. Summary and discussion are given in Section 5. Proofs are given in the Appendix.

\textbf{Notation}:
 For $0< q\leq\infty$,
we write $\| \beta\|_{q}:=(\sum_{i=1}^{p}|\beta_{i}|^{q})^{1/q}$ as the $\ell_{q}$-norm
of a $p-$dimensional vector $\beta$. If $q=\infty$, we have $\| \beta\|_{\infty}:=\max\limits_{i=1,...,p}|\beta_{i}|$. For a subset $C\subset \R^n$, let $C^0$ and $\overline{C}$ denote the interior and closure of $C$ in $\R^n$,
respectively. Let $\Omega(\mathbf{x}, r)$ denote the open ball $\{\mathbf{y}: \|\mathbf{x}-\mathbf{y}\|< r \}$,
and $\overline{\Omega(\mathbf{x}, r)}$ be its closure.

\section{Null models for undirected network data}\label{section 2}
\subsection{Several Models}
Following \cite{yanxu2013}, we consider an undirected graph $\mathcal{G}_{n}$ on n ( $n \geq 2$ ) agents labeled by $1,...,n$. Let $a_{ij}\in\{0,1\}$ be the weight of the
undirected edge between $i$ and $j$. That is, if there is a link between $i$ and $j$, then $a_{ij}=1$;
otherwise, $a_{ij}=0$. Denote $A=(a_{ij})_{n \times n}$ as the symmetric adjacency matrix of $\mathcal{G}_{n}$.
We assume that there are no self-loops, i.e., $a_{ii}=0$. Define $d_{i}=\sum_{j\neq i}^{n}a_{ij}$ as the degree of vertex $i$ and $\mathbf{d}=(d_{1},\dots,d_{n})^{\top}$ as the degree sequence of the graph $\mathcal{G}$. 
We suppose that the adjacency matrix $A$ has independent Bernoulli elements such that $\pr (a_{ij}=1) = p_{ij}$ and specify the corresponding family of probability null models for $A$. Here $\boldsymbol{\alpha}^{*}=(\alpha_{1}^{*},\alpha_{2}^{*},...,\alpha_{n}^{*})^{\top}$ is a parameter vector. The parameter
$\alpha_i^*$ quantifies the effect of the vertex $i$. To this end, let $\varepsilon(\cdot,\cdot): \R^2 \mapsto \R$ be a smooth bivariate function satisfying $ \varepsilon(x,y) = \varepsilon(y,x)$. Consider the model $\emph{M}_{\varepsilon}$ specified $p_{ij}$ as
\begin{equation}\label{2.1}
\emph{M}_{\varepsilon}: \log p_{ij} = \alpha_{i}^{*} + \alpha_{j}^{*} + \varepsilon(\alpha_{i}^{*}+\alpha_{j}^{*})
\end{equation}
so that we obtain a class of log-linear models indexed by $\varepsilon$. In fact, this class encompasses three common choices of link functions:
\begin{equation}\label{2.2:1}
\emph{M}_{\log}:  \log p_{ij} = \alpha_{i}^{*} +  \alpha_{j}^{*},
\end{equation}
\begin{equation}\label{2.2:2}
\emph{M}_{\rm logit}:  {\rm logit} (p_{ij})= \alpha_{i}^{*} +  \alpha_{j}^{*},~\text{where}~{\rm logit} (x)= \begin{matrix} \frac{1}{1 + e^{-x}} \end{matrix},
\end{equation}
\begin{equation}\label{2.2:2}
\emph{M}_{\rm cloglog}:  \log(-\log(1-p_{ij}))= \alpha_{i}^{*} + \alpha_{j}^{*}.
\end{equation}
To see this, set $\varepsilon(\alpha_{i}^{*},\alpha_{j}^{*}) = -\log\{1+ \exp(\alpha_{i}^{*}+ \alpha_{j}^{*})\}$ for the model $\emph{M}_{logit}$. As we have seen, the logit-link model $\emph{M}_{logit}$ is an undirected version of \cite{Holland1981An} exponential family random graph model without reciprocal parameter. As noted by \cite{yanxu2013}, the degree sequence of $\mathcal{G}$ is sufficient for $\alpha$ in this case, and they derived its asymptotic normality. The log-link model $\emph{M}_{\log}$ can be considered as an undirected version of expected degree model with $\varepsilon(\alpha_{i}^{*},\alpha_{i}^{*}) = 0$ constructed by \cite{Fan2002Connected}. Setting $\varepsilon(\alpha_{i}^{*},\alpha_{j}^{*}) = \log\{1- \exp(\alpha_{i}^{*}+ \alpha_{j}^{*})\}-(\alpha_{i}^{*}+ \alpha_{j}^{*})$, we get the complementary log-log link model $\emph{M}_{\rm cloglog}$ (see \cite{Mccullagh1989Generalized}).

\subsection{The sub-Exponential/-Gamma noisy sequence}
In this section, we will prepare the probability and distribution preliminary for later network analysis. This section can be divided into two parts. In the first part, we are going to recap the definition of a specific type of distributions and state their basic properties. 

\begin{definition}[Sub-Gaussian distribution]
 A random variable $X \in \mathbb{R}$ with mean zero is sub-Gaussian with
variance proxy $\sigma^{2}$ if its MGF satisfies
$$\mathrm{E}[\exp (s X)] \leq \exp \left(\frac{\sigma^{2} s^{2}}{2}\right), \quad \forall s \in \mathbb{R}.$$
In this case we write $X \sim \operatorname{subG}\left(\sigma^{2}\right) .$
\end{definition}
\noindent Similarly, we define that a random variable is called sub-exponential if its survival function is bounded by that of a particular exponential distribution.
\begin{definition}[Sub-exponential distribution]\label{def:Sub-exponential}
 A random variable $X \in \mathbb{R}$ with mean zero is sub-exponential with
parameter $\sigma^{2}$ (denoted by $X \sim \operatorname{subE}(\lambda))$ if its MGF satisfies
\begin{equation}\label{eq-sub-exponentialmoments}
 {\mathrm{{E}}}e^{s X} \leq e^{\frac{s^{2}\lambda ^{2}}{2}} \quad
\text{for all } |s| <\frac{1}{\lambda}.
\end{equation}
\end{definition}
\noindent Obviously, sub-Gaussian random variables are sub-exponential but not vice verse. And there are some equivalent definitions of sub-exponential distributions, see \cite{Rigollet2019}, which can lead to sub-exponential norm used in concentration inequalities. The equivalent definitions and other details can be seen in the Appendix.

\begin{definition}[Sub-exponential norm]\label{def: sub-exponential}
The sub-exponential norm of $X$, denoted $\|X\|_{\psi_1}$,
  is defined as
  \begin{equation}\label{eq: psione}
  \|X\|_{\psi_1} = \inf \left\{ t>0 :\; \E \exp(|X|/t) \le 2 \right\}.
  \end{equation}
\end{definition}

The above norm provides us with a useful tool to connect MGF and the defined norm, and hence makes it possible to give concentrations for the sub-exponential variables. The following lemma confirms Definition \ref{def: sub-exponential} would give a concise form of concentrations.

\begin{lemma}[Properties of sub-exponential norm]\label{prop:Psub-E}
If $\E \exp(|X|/\|X\|_{{\psi _1}} ) \le 2$, then we have
\begin{enumerate}[\rm{(}a\rm{)}]
\item Tail bounds
\begin{equation}\label{sub-exponential tail}
\pr \{ |X| > t \} \leq 2\exp(-t/\|X\|_{\psi _{1}}) \quad \text{for all } t \geq 0;
\end{equation}	
\item  Moment bounds
$$\E |X|^k \le 2{\|X\|_{{\psi _1}}^k}k! \quad \text{for all integer}~k \ge 1;$$
\item If $\E X=0$, we get the MGF bounds
\begin{equation}\label{sub-exponentialmgf}
\E e^{s X} \leq e^{({\rm{2}} \left\| X \right\|_{{\psi _1}})^{2}s^{2}}~\text { for all }|{s}|<1/(2\left\| X \right\|_{{\psi _1}})
\end{equation}
which gives $X  \sim \operatorname{subE}(2\left\| X \right\|_{{\psi _1}})$.
\end{enumerate}
\end{lemma}
Lemma \ref{prop:Psub-E}(c) implies that the following user-friendly concentration inequality would contain all known constant. One should note that Theorem 2.8.1 of \cite{Vershynin2018} includes an unspecific constant, so it is inefficacious when constructing non-asymptotic confident interval for sub-exponential sample mean.

\begin{corollary}[Concentration for sub-exponential sum of r.v.s,~\cite{zhang2020concentration}]\label{Sub-expConcentration}
Let $\{ X_{i}\} _{i = 1}^n $ be zero mean independent sub-exponential distribution with $\|X_i\|_{\psi_1}\le \infty$. Then for every $t \ge 0$,
\[\pr ( {| {\sum\limits_{i = 1}^n {{X_i}} }| \ge t} ) \le 2 \exp\{  { - \frac{1}{4}( {\frac{{{t^2}}}{\sum\nolimits_{i = 1}^n {2\|X_i\|_{\psi_1}^2}} } \wedge \frac{t}{\mathop {\max }\limits_{1 \le i \le n}\|X_i\|_{\psi_1}}}) \}.\]
\end{corollary}

\cite{Hay2009} and \cite{Karwa2016} used the Laplace mechanism to provide privacy protection in which independent and identically distributed Laplace random variables are added into the input data.
Here, we consider a general distribution for the noisy variables with the Laplace distribution as a special case.
\begin{example} [Laplace r.vs]
{\color{black}{ A r.v. $X$ follows a Laplace distribution (${{\textrm {Laplace}}(\mu ,b)},~\mu \in \mathbb{R} ,b>0$) if its probability density function is
$f(x )=\frac{1}{2 b} e^{-\frac{|x-\mu|}{b}}$. The Laplace distribution is a distribution of the difference of two independent identical exponential distributed r.vs, thus it is also sub-exponential distributed by using Proposition \ref{sub-exponentialConcentration}(a). The graph of $f(x)$ are like two exponential distributions which are spliced together back-to-back.}}
\end{example}

\begin{example} [Geometric distributions]
{\color{black}{The \emph{geometric distribution} $X\sim \mathrm{Geo} (q) $ for r.v. $X$ is given by
$\pr (X=k) ={(1 - q)} {q^{k-1}},~~q \in (0,1),k=1,2,\cdots.$ The mean and variance of ${\rm{Geo}} (q) $ are $\frac{{1 - q}}{{q}}$ and $\frac{{1 - q}}{{{{q}^2}}}$ respectively. Apply
Lemma 4.3 in \cite{hillar2013maximum}, we have ${({\rm{E}}|X{|^k})^{1/k}} < \frac{{2k}}{{ - \log (1 - q)}}$. Followed by the triangle inequality applied to the $p$ -norm and Jensen's inequality for $k \ge 1$, we have
$$
\mathrm{E}[|X-\mathrm{E}X|^{k}]^{1 / k} \le \mathrm{E}[|X|^{k}]^{1 / k}+|\mathrm{E}[X]| \leq 2 \mathrm{E}[|X|^{k}]^{1 / k}\le \frac{{4k}}{{ - \log (1 - q)}}.
$$
Then Proposition \ref{prop: sub-exponential properties}(3) shows that the \textit{centralized geometric distribution} is sub-exponential with $K_3=\frac{{4}}{{ - \log (1 - q)}}$.
}}
\end{example}

\begin{example}[Discrete Laplace r.vs]
A r.v. $X$ obeys the discrete Laplace distribution with parameter $p \in(0,1),$ denoted by $\mathrm{DL}(p),$ if
\[
f_{p}(k)=\mathrm{P}(X=k)=\frac{1-p}{1+p} p^{|k|}, \quad k \in \mathbb{Z}=\{0,\pm 1,\pm 2, \ldots\}.
\]
Similar to Laplace distribution, the discrete Laplace r.v. is the difference of two independent identical geometric distributed r.vs (see Proposition 3.1 in \cite{Inusah2006A}). Since geometric distribution is sub-exponential in the previous example, the Proposition \ref{sub-exponentialConcentration}(a) implies that discrete Laplace is also sub-exponential distributed. In differential privacy of network models, the noises are assumed from the discrete Laplace distribution (see \cite{fan2020} and references therein).
\end{example}

In statistical applications, we sometimes do not expect the bounded assumption in Hoeffding's inequality, the following Bernstein's inequality for a sum of independent random variables allows us to estimate the tail probability by a weaker version of exponential condition on the growth of the $k$-moment (like a condition of the exponential MGF) without any assumption of boundedness.
\begin{lemma}[Bernstein's inequality]\label{lm-Bernsteingm}
The centred independent random variables $X_1,\ldots, X_n$ satisfy \textit{the growth of moments condition}
\begin{equation}\label{eq-Bernsteinmoments}
{\rm{E}}{\left| {{X_i}} \right|^k} \le \frac{1}{2}\nu_i^2{\kappa_i^{k - 2}}k!,~(i = 1,2, \cdots ,n),~\text{for all}~k\ge 2
\end{equation}
where $\{\kappa_i\}_{i=1}^n$ and $ \{\nu _i\}_{i=1}^n$ are constants independent of $k$. Denote ${\nu _n^2} = \sum\limits_{i = 1}^n {v _i^2} $ (the fluctuation of sums) and $\kappa= \mathop {\max }\limits_{1 \le i \le n} \kappa_i$. Then we have $
{\rm{E}} e^{s X_{i}} \leq e^{s^{2} \nu_{i}^{2} /(2-2 \kappa_i|s|)}.
$ And for $t>0$
\begin{equation}\label{eq-Bernstein}
\pr \left( {\left| {{S_n}} \right| \ge t} \right) \le 2\exp \left( { - \frac{{{r^2}}}{{2{\nu _n^2} + 2\kappa t}}} \right),~~\pr( {\left| {{S_n}} \right| \ge \sqrt {2\nu _n^2t}  + \kappa r}) \le 2{e^{ - t}}.
\end{equation}
\end{lemma}
\noindent The proof Bernstein's inequality for the sum of independent random variables can be founded in p119 of \cite{Gin2015Mathematical}.

Like the sub-Gaussian, \cite{boucheron2013concentration} defines the sub-Gamma r.v. based on the right tail and left tail with variance factor $v$ and scale factor $b$.
\begin{definition}[Sub-Gamma r.v.]\label{def: sub-Gamma}
A centralized r.v. $X$ is {\em sub-Gamma} distributed with variance factor $\upsilon>0$ and scale parameter $c>0$ (denoted by $X \sim \mathrm{sub}\Gamma(\upsilon,c)$) if
\begin{equation}\label{eq:sub-gamma}
\log (\mathrm{{E}}{e^{s{X}}}) \leq \frac{s^{2}}{2} \frac{\upsilon}{1-c|s|}, \quad \forall~0<|s|<c^{-1}.
\end{equation}
\end{definition}

The sub-exponential  moment condition  \eqref{eq-sub-exponentialmoments} would imply that Bernstein's moment condition \eqref{eq-Bernsteinmoments} is observed as
$$\log (\mathrm{{E}}{e^{s{X}}}) \le {\frac{s^{2}\lambda ^{2}}{2}} \le \frac{s^{2}\lambda ^{2}}{
{2(1 - \lambda|s|)}},~\forall~|s| <\frac{1}{\lambda}.$$

\begin{lemma}[Concentration for sub-Gamma sum, Section 2.4 of \cite{boucheron2013concentration}]\label{sub-GammaConcentration}
Let $\{ X_{i}\} _{i = 1}^n $ be independent $\{\mathrm{sub}\Gamma(\upsilon_i,c_i)\} _{i = 1}^n$ distributed with zero mean. Define $c= \mathop {\max }\limits_{1 \le i \le n} c_i$, then
\begin{enumerate}[\rm{(}a\rm{)}]
\item Closed under addition: $ S_n:=\sum\limits_{i = 1}^n {{X_i}}  \sim {\rm{sub}}\Gamma({\sum\limits_{i = 1}^n\upsilon_i},c)$;
\item $\pr (|S_n|\geq t)\leq 2 \exp \left(-{\frac {t^{2}/2}{\sum _{i=1}^{n}\upsilon_i+ct}}\right)$ and $\pr \{|S_n|>(2t \sum_{i=1}^{n}\upsilon_i)^{1/2}+c t\} \leq 2 e^{-t},~\forall~t \ge 0$;

\item If $X \sim \mathrm{sub}\Gamma(\upsilon,c)$, the even moments bounds satisfy
$\E{X^{2k}} \le k!{({8v} )^k}+(2k)!{({4c} )^{2k}},~k\ge 1.$

\end{enumerate}
\end{lemma}
The concentration inequalities introduced in above only concerns the linear
combinations of independent random variables. For lots of
applications in high-dimensional statistics, we have
to control the maximum of the n r.vs when deriving error bounds for the proposed estimator.
In our proof of Theorem \ref{theorem:1}, the following  maximum inequality is crucial.
\begin{lemma}[Concentration for maximum of sub-Gamma
random variables]\label{sub-GammaMAX}
Let $\{ X_{i}\} _{i = 1}^n $ be independent $\{\mathrm{sub}\Gamma(\upsilon_i,c_i)\} _{i = 1}^n$ distributed with zero mean.  Denote $\max\limits_{i=1,...,n}\upsilon_{i} = \upsilon$ and $\max\limits_{i=1,...,n}c_{i} = c$, we have
\begin{equation}\label{eq:max}
\E (\max_{i=1,\ldots,n}|X_i|)  \leq \sqrt{2\upsilon \log (2n)} + c  \log (2n).
\end{equation}
\end{lemma}

\section{Estimation and its asymptotic properties}\label{section:3}
In this section, we will derive the asymptotic results for the estimator with an increasing sub-Gamma degree sequence. 
Note that $\E(a_{ij})$ only depends on the $e^{\alpha_{i}^{*}+ \alpha_{j}^{*}+\varepsilon_{ij}(\alpha_{i}^{*}+\alpha_{j}^{*})}$.
Let $\mathbf{d}=(d_{1}, \ldots, d_{n})^{\top}$ be the degree sequence of graph $\mathcal{G}_{n}$.
We assume that random variables $\{ e_{i}\}_{i=1}^n$ are mutually independent and distributed in sub-gamma distributions $\{{\rm{sub}}\Gamma(\upsilon_i,c_i)\} _{i = 1}^n$ with respective parameters $\{(\upsilon_i,c_i)\} _{i = 1}^n$. Then we observe the noisy sequence $\tilde{d}$ instead of $d$, where
\begin{equation}\label{equation:df}
\begin{array}{lcl}
\tilde{d}_{i} & = & d_{i} + e_{i},~~i=1, \ldots, n. \\
\end{array}
\end{equation}
We use moment equations to estimate the degree parameter with the noisy sequence $\tilde{d}$ instead of $d$.
Define a system of functions:
\begin{gather*}
    F_i(\bs{\alpha}) := \tilde{d_{i}} - \E(d_{i}) =\tilde{d_{i}} - \sum_{j\neq i}^n e^{\alpha_{i} + \alpha_{j} + \varepsilon_{ij}(\alpha_{i} + \alpha_{j})},~~i=1,\ldots, n, \\
    F(\bs{\alpha}) = \big(F_1(\bs{\alpha}), \ldots, F_n(\bs{\alpha}) \big)^\top.
\end{gather*}
Now, we define our estimator $\widehat{\bs{\alpha}}$ as the solution to the equation $F(\bs{\alpha})=0$, i.e.,
\begin{equation}\label{estimator_def}
    \widehat{\bs{\alpha}} := \{ \bs{\alpha} : F(\bs{\alpha})=0 \}
\end{equation}
It is not hard to see that the estimator is actually induced by the moment equation $\mathbf{\tilde{d}} = \E (\mathbf{d})$.
\vspace{2ex}

These asymptotic results of $\widehat{\bs{\alpha}}$ hold for all $\varepsilon$ satisfying the following condition.
\begin{assumption}\label{assumption:2}
For all pairs of node $i$ and node $j$, all choices of $k,l$ and $m$ $(k,l,m=1,...,n)$, the function $\varepsilon_{i,j}$, $\partial \varepsilon_{i,j} / \partial \alpha_{k} $, $\partial \varepsilon^{2}_{i,j} / \partial \alpha_{k} \partial \alpha_{l} $, and $\partial \varepsilon^{3}_{i,j} / \partial \alpha_{k}^{*} \partial \alpha_{l}  \partial \alpha_{m} $, are sub-exponential in $\alpha_{i} +\alpha_{j} $. That is, there exists a constant $M_{0}$ such that the absolute values of these functions are bounded by $M_{0} \exp (\alpha_{i} + \alpha_{j})$.
\end{assumption}
Note that the solution to the equation $F({\boldsymbol{\alpha}})=0$ is precisely the moment estimator. Here, we consider the symmetric parameter space
\[D = \{ {\boldsymbol{\alpha}} \in \mathbb{R}^{n}: -Q_{n}\leq \alpha_{i} + \alpha_{j} \leq Q_{n}, \ Q_{n} >0, \ 1\leq i<j \leq n\}.\]
The uniform consistency of $\widehat{\boldsymbol{\alpha}}$ and asymptotic distribution of the parameter estimator are stated as follows, and the proofs are given in Appendix.

\begin{theorem}[Consistency]\label{theorem:1}
If
\begin{equation*}
    e^{Q_{n}+e^{Q_{n}}} = \left(\frac{\sqrt{(n-1)\log(n-1)}+\sqrt{2\upsilon \log (2n)} + c  \log (2n)}{n} \right)^{-1/34}
\end{equation*}
and  $\max_{i=1,...,n}\upsilon_{i} = \upsilon$, $\max_{i=1,...,n}c_{i} = c$, then as $n \rightarrow \infty$, the estimator $\widehat{\bs{\alpha}}$ exists and satisfies
\begin{equation}\label{consistency-unified}
\|\bs{\widehat{\alpha}} - \bs{\alpha}^{*} \|_\infty = O_p \left( \frac{1}{n}e^{15Q_{n}+e^{Q_{n}}} \big(\sqrt{n \log n }+\sqrt{2\upsilon \log (2n)} + c  \log (2n) \big) \right) = o_p(1).
\end{equation}
\end{theorem}

We use the Newton-Kantovorich theorem to prove the consistency of the estimators by constructing the Newton iterative sequence. This technical step is different from Chatterjee et al. (2011) and provides a simple proof.  The proof of the theorem is in Appendix.

\begin{theorem}[Asymptotic normality]\label{theorem:2}
Using the same notations as Theorem \ref{theorem:1}, if
$$\frac{1}{n}e^{45Q_{n}+e^{Q_{n}}} \left(\sqrt{(n-1)\log(n-1)}+\sqrt{2\upsilon \log (2n)} + c\log (2n)\right)^2 = o(n^{1/2})$$
then for any fixed $k\ge 1$, as
$n \to\infty$, 
 the vector consisting of the first $k$ elements of
$(B^{-1})^{1/2}(\widehat{\bs{\alpha} }-\bs{\alpha}^{*} )$ is asymptotically distributed as $N(0,\mathbf{I}_k)$, 
where $(B^{-1})^{1/2}=\operatorname{diag}(v_{11}^{1/2}, \ldots,
v_{nn}^{1/2})$ with
\begin{equation*}
    v_{ii}=\sum_{j \neq i}^{n}e^{\alpha_{i}^{*}+\alpha_{j}^{*}+\varepsilon_{i,j}(\alpha_{i}^{*},\alpha_{j}^{*})}\left(1+\frac{\partial \varepsilon_{ij}(\alpha_{i}^{*},\alpha_{j}^{*}) }{\partial \alpha_{i}^{*} } \right).
\end{equation*}
\end{theorem}

The proof of the theorem is in Appendix.

\section{Numerical studies}\label{section 4}
In this section, we will evaluate the asymptotic result in  Theorems \ref{theorem:2} through numerical simulations and a real data example. The simulation will be conducted using Hermite distributions which we have introduced and proved its properties under the framework of Section 2.

\subsection{Simulation studies}
For the simulation study of all models $M_{\varepsilon}$, we set $\alpha_{i}^{*} = i*L/n$ for $i=1,...,n$. Other parameter settings in simulation studies are listed as follows. We consider three different values $-\log(\log(n))^{1/3}$, $ -\log(\log(n))^{1/2}$, and $-\log(\log(n))$ for $L$ in case of model $M_{log}$. For model $M_{logit}$, we consider three different values $L=0$, $\log(\log(n))$ and $\log(n)^{1/2}$. In case of model $M_{cloglog}$, we set three different values $L=\log(\log(\log(n)))^{3/2}, \log(\log(\log(n)))^{5/4}$ and $\log(\log(\log(n)))$. The noise $e_{i}(i=1,...,n)$ is the difference between two independent and identically distributed Hermite distributions (see Definition \ref{herm}). We consider three cases of Hermite distribution where parameters $a_1=0.01$, $a_2=(\Lambda-0.01)/4, m=2$; $a_1=\Lambda-0.01$, $a_2=0.025$, $m=2$ and $a_1=4*\Lambda/5$, $a_2=\Lambda/5$, $m=2$. Here, we set $\Lambda = 2*\exp(-\lambda_{0}/2)/(1-\exp(-\lambda_{0}/2))^{2}$ and $\lambda_{0}=2$.
We consider two values for $n=100$ and $n=200$. Note that by Theorems \ref{theorem:2}, $\widehat{\xi}_{ij} =(\widehat{\alpha}_i + \widehat{\alpha}_j-(\alpha_{i}^{*} + \alpha_{j}^{*}))/(1/\widehat{v}_{ii}+1/\widehat{v}_{ii})^{1/2}$ is an asymptotically normal distribution, where $\hat{v}_{ii}$ is the estimator of $v_{ii}$ by replacing $\alpha_i$ with $\hat{\alpha}_i$.
The quantile-quantile(QQ) plots of $\xi_{ij}$ are drawn. We ran $10000$ simulations for each scenario.

We simulate with $n=100$, $n=200$, three values for $L$ and three cases of Hermite distribution to find that the QQ-plots for each combination are similar. In order to save the place, we only present the QQ plots of three model in Figure \ref{Fig1:log}, Figure \ref{Fig2:logit} and Figure \ref{Fig3:cloglog} when $n=200$, $a_1=4*\Lambda/5$, $a_2=\Lambda/5$, and $m=2$ for each case. The horizontal and vertical axes are the theoretical and empirical quantiles respectively, and the straight lines correspond to the reference line $y=x$. In Figure \ref{Fig1:log}, we first observe that the empirical quantiles agree well with the ones of the standard normality of $\widehat{\xi_{ij}}$, expect for pair $(n/2, n/2+1)$ and $(n-1, n)$ when $L=-\log(\log(n))$ in the case of model $M_{log}$. For the case of model $M_{logit}$, there are no notable derivations from the standard normality for each scenario in Figure \ref{Fig2:logit}. We also observe that the empirical quantiles agree well with the ones of the standard normality of $\widehat{\xi}_{ij}$, expect for pair $(n/2, n/2+1)$ and $(n-1, n)$ when $L=\log(\log(\log(n)))$ in Figure \ref{Fig3:cloglog}.

The coverage probability of the $95\%$ confidence interval for $\alpha_i-\alpha_j$, the length of the confidence interval and the frequency that the MLE did not exist are reported in Table \ref{tab:log:1}, Table \ref{tab:logit:2} and Table \ref{tab:cloglog:3}. We can see that the length of estimated confidence interval increases as $L$ increases for fixed $n$, and decreases as $n$ increases for fixed $L$.

\subsection{ A data example}
We use the KAPFERER TAILOR SHOP network dataset created by Bruce Kapferer which is downloaded from \url{http://vlado.fmf.uni-lj.si/pub/networks/data/ucinet/ucidata.htm}. In each study they obtain measures of social interaction among all actors. In this network data, $``1''$ represents they have a friend relationship between two actors, otherwise, it is denoted as $``0''$. Because the estimate $\hat{\alpha}$ does not exist when the degree is zero, we remove the vertex 17 and 22 whose degree is zero before analysis, so the network with the left 37 vertices in each table remains.
When the parameters of three kinds of noise distribution function $e_{i}(i=1,...,n)$ are set, the parameter estimation is similar in case of model $M_{logit}$, $M_{log}$ and $M_{cloglog}$.
Thus we here only show the parameter estimation under the case of noise distribution function parameter for $a_1=4*\Lambda/5$, $a_2=\Lambda/5$ and $m=2$. From Figure \ref{Fig4}, we first observe the scatter plots of the noisy degree sequence $\tilde{d}$ corresponding to the parameter estimation $\hat{\alpha}$ under model $M_{\varepsilon}$ and the value of $\hat{\alpha}$ increases as the number of $\tilde{d}$ climbs. The larger the estimated parameters $\hat{\alpha}$, the actors have more friends. For these three cases of model $M_{\varepsilon}$, the estimated parameters and their standard errors as well as the $95\%$ confidence intervals and the size of noisy degree sequences are reported in Table \ref{tab:example:log}, Table \ref{tab:example:logit} and Table \ref{tab:example:cloglog} respectively.
The value of estimated parameters reflects the corresponding size of noisy degrees. For example, the large five degrees are $26,17,16,15,14$ for vertices $16,18,32,11,12$ which also have the top five influence parameters at $0.36,-0.05,-0.12,-0.18,-0.25$. On the other hand, the five vertices with smallest parameters $-2.89,-2.19,-1.79,-1.28,-1.10$ have degrees at $1,2,3,5,6$. In Table \ref{tab:example:logit} and Table \ref{tab:example:cloglog},
 the larger the parameter $\hat{\alpha}$, the greater the degree $\tilde{d}$ of the node. This is the same as the conclusion in Table \ref{tab:example:log}.

\section{Summary and discussion} \label{section 5}
In this paper, we release the degree sequences of the class of binary networks under the sub-Gamma noisy mechanism.
We establish the asymptotic result including the consistency and asymptotically normality of the parameter estimator when the number of parameters goes to infinity. By using the
Newton-Kantorovich theorem, we try to ignore adding noisy process and obtain the existence
and consistency of the parameter estimator satisfying equation $\mathbf{\tilde{d}} = \E (\mathbf{d})$. Furthermore, we give some simulation results to illustrate that the asymptotic normality behaves well under model $\emph{M}_{\log}$, $\emph{M}_{logit}$ and $ \emph{M}_{cloglog}$. However, an edge in networks takes not only binary values but also weighted edges in many scenarios. We will investigate null models for these directed weighted networks in the future. It is worth noting that the conditions imposed on $Q_{n}$,$a_1$, $a_2$, and $m=2$ may not be the most possible. In particular, the conditions guaranteeing the asymptotic normality are stronger than those guaranteeing the consistency. Simulation studies suggest that the conditions on $Q_{n}$ might be relaxed. It can be noted that the asymptotic behavior of the parameter estimator depends not only on $Q_{n}$ $a_1$, $a_2$, and $m=2$, but also on the configuration of all the parameters, We will investigate this in future studies.

In this paper we derived individual parameter asymptotic properties, and we can also study on a linear combination of all the parameter estimation in binary networks with noisy degree sequence in the future work. In our paper, we only consider the model heterogeneity parameter. In network data, the second distinctive feature inherent in most natural networks is the homophily phenomenon. \cite{Yan2019} established the uniform consistency and asymptotic normality of the heterogeneity parameter and homophily parameter estimators. On the other hand, sub-Weibull variables, as an extension of sub-Gamma variables, enable variables have heavier tails, which may also be consider in networks models. Fortunately, there have been some articles investigating concentration of sub-Weibull variables, see \cite{zhang2021sharper} for instance. And we further investigate a central limit theorem
for a linear combination of all the maximum likelihood estimators of degree parameter when the number of nodes goes to infinity[\cite{luo2020asymptotic}]. And the asymptotic theory of the affiliation network model with noise sequence is also worth further study[\cite{luo2021affiliation}]. We will investigate these aspects in future studies. \par

\section{Appendix}

\subsection{Two-side discrete compound Poisson and Hermite distributions}
The negative binomial random variable belongs to the exponential family when the dispersion parameter is known. But if the parameter in negative binomial random variable $X$ is unknown in real world problems \citep{zhang2022elastic}, it does not belong to the exponential family. Whereas, it is well-known that Poisson and negative binomial distributions belong to the family of discrete infinitely divisible distributions (also named as discrete compound Poisson distributions); see \cite{zhang2014notes} and the references therein.

It should be noted that the discrete Laplace random variable is the difference of two i.i.d. geometric distributed random variables (see Proposition $3.1$ in \cite{Inusah2006A}). The geometric distribution as a class of infinitely divisible distribution is a special case of discrete compound Poisson distribution. The difference of geometric noise-addition mechanism can be flexibly extended to the difference between two i.i.d. (or independent) discrete compound Poisson random variables (see Definition $4.2$ of \cite{zhang2016characterizations}). In fact, the difference between two independent discrete compound Poisson random variables follows the infinitely divisible distributions with integer support; see Chapter IV of \cite{steutel2003infinite}.

\begin{definition}\label{def-dcp}
We say that $Y$ is discrete compound Poisson (DCP) distributed if the
characteristic function of $Y$ is
\begin{equation}
\label{eq:cf} {\varphi _{Y}}(t) = \mathrm{{E}}
{e^{\mathrm{{i}}tY}} = \exp \Biggl\{ \sum_{k= 1}^{\infty }{{
\alpha _{k}}\lambda \bigl({e^{\mathrm{{i}}t
k}} - 1 \bigr)} \Biggr\}
\quad ( t \in \mathbb{R}),
\end{equation}
where $({\alpha _{1}}\lambda ,{\alpha _{2}}\lambda , \ldots )$ are
infinite-dimensional parameters satisfying $\sum_{i = 1}^{\infty }
{\alpha _{k}} = 1$, ${\alpha _{k}} \ge 0$, \mbox{$\lambda > 0$}. We denote it as
$Y \sim \operatorname{DCP}({\alpha _{1}}\lambda ,{\alpha _{2}}\lambda ,
\ldots )$.
\end{definition}
Based on the \eqref{eq:cf}, the discrete compound Poisson random variable $X_i$ has the weighted Poisson decomposition
\begin{equation*}
X:= \sum_{k = 1}^{\infty }{k{N_{k}}},~~\text{where }{N_{k}}\overset{ind.}{\sim} {\rm Poisson}({\alpha _{k}}\lambda),
\end{equation*}
so we have $\E X = \lambda \sum\limits_{k = 1}^\infty  k {\alpha _k}$.

\cite{prekopa1952composed} first considered the difference between two composed Poisson distributions (we name it two-side composed Poisson distribution in the following context), of which the characteristic function is
\begin{equation}\label{cf-1}
    \varphi(t)=\exp \left\{\sum\limits_{k\in \mathbb{Z}\backslash\{0\}}C_{k}(e^{i\lambda_{k}t}-1)\right\},
\end{equation}
where $C_{k},\lambda_{k}\ge 0$ and $\sum\limits_{k\in \mathbb{Z}\backslash\{0\}}C_{k}<\infty$. We notice \eqref{cf-1} degenerates to \eqref{eq:cf} when $\lambda_{k}=k \,(k\ge 1)$ and $C_{k}=0 \,(k\le 0)$.

Notice the characteristic function of Levy process $Z(t)$ can be represented by
\begin{equation*}
    \varphi(\theta)=\operatorname{E}[e^{i\theta Z(t)}]=
    \exp\left\{ ait\theta-\frac{1}{2}\sigma^{2}t\theta^{2}+t\int_{\mathbb{R}\backslash 0}(e^{i\theta x}-1-i\theta x \mathbb{I}_{|x|<1})w(dx) \right\}
\end{equation*}
using Levy-Khinchine formula, where $a\in \mathbb{R},\, \sigma>0$ and $\mathbb{I}_{|x|<1}$ as the indicator function. $w$ is usually referred as the Levy measure, which is a non-negative measure that satisfies $\int_{\mathbb{R}\backslash 0}\min\{x^{2},1\}w(dx)<\infty$.

Construct the following Levy measure $w(d x)=\alpha_{k}\lambda d \delta_{k}$, where $\sum\limits_{k\in \mathbb{Z}\backslash\{0\}}\alpha_{k}=1$ and $\delta_{k}$ is the Dirac measure. According to the definition of $w(dx)$, we have
\begin{equation*}
    ait\theta-\frac{1}{2}\sigma^{2}t\theta^{2}=t\int_{\mathbb{R}\backslash 0}(i\theta x \mathbb{I}_{|x|<1})w(dx)=0.
\end{equation*}
Therefore, the characteristic function of two-side CPD denoted by $Z$ is
\begin{equation}\label{cf-2}
    \operatorname{E}[e^{itZ}]=\exp \left\{\lambda \sum\limits_{k=1}^{\infty}\alpha_{k}(e^{ikt}-1)+\mu\sum\limits_{k=1}^{\infty}\beta_{k}(e^{-ikt}-1) \right\},
\end{equation}
which can also be used as the definition of two-side CPD. We say a random variable has two-side CPD if the characteristic function of it satisfies the \eqref{cf-2}.

Notice the discrete composed Poisson distribution can be generated by Poisson process. Similarly, the two-side CPD can be generated by the difference between two Poisson processes with different parameters. We introduce the two-side Poisson distribution later.

When the characteristic function of random variable $Z$ satisfies \eqref{cf-2} with $\alpha_{1}=\beta_{1}=1$ and $\alpha_{k}=\beta_{k}=0\,(k>1)$, we say $Z$
satisfies tow-side Poisson distribution, of which the characteristic function is $\exp\{\lambda(e^{i t}-1)+\mu(e^{-it}-1)\}$ and the p.d.f. is in the form
\begin{equation*}
    \pr (Z=k)=e^{-(\lambda+\mu)}\left( \frac{\lambda}{\mu} \right)^{k/2}I_{|k|}(2\sqrt{\lambda\mu}), \qquad k\in \mathbb{Z},
\end{equation*}
where $I_{n}(x)$ is the modified Bessel function of the first kind.
$I_{n}(x)=\sum\limits_{k=0}^{\infty}\frac{1}{k!(k+n)!}\left(\frac{x}{2} \right)^{2k+n}$ satisfies: (i) the expansion $\exp\{\frac{1}{2}x(z+z^{-1})\}=\sum\limits_{n\in \mathbb{Z}}I_{n}(x)z^{n}; $ 
(ii) $I_{n}(x)=I_{-n}(x)$.

Two special cases occur when we set $\mu=0$ and $\lambda=0$ in a two-side Poisson distribution $Z$ respectively. When $\mu=0$, $Z$ degenerates to Poisson distribution, consistent with the fact that
\begin{equation*}
    \pr (Z=k)=\lim\limits_{\mu\rightarrow 0}\left(\frac{\lambda}{\mu}\right)^{k/2}\sum\limits_{i=0}^{\infty}\frac{(\sqrt{\lambda\mu})^{2i+|k|}}{e^{\lambda+\mu}i!(i+|k|)!}
    =\frac{\lim\limits_{\mu\rightarrow0}\left(\frac{\lambda}{\mu}\right)^{k/2}(\sqrt{\lambda\mu})^{|k|}}{e^{\lambda}|k|!}=\frac{\lambda^{k}e^{-\lambda}}{k!}\,(k\ge 0).
\end{equation*}
When $\mu=0$, we derive the p.d.f. of $Z$ that
\begin{equation*}
    \pr (Z=k)=\lim\limits_{\lambda\rightarrow 0}\left(\frac{\lambda}{\mu}\right)^{k/2}\sum\limits_{i=0}^{\infty}\frac{(\sqrt{\lambda\mu})^{2i+|k|}}{e^{\lambda+\mu}i!(i+|k|)!}
    =\frac{\lim\limits_{\lambda\rightarrow0}\lambda^{k/2+|k|/2}\mu^{-|k|/2}u^{|k|/2}}{e^{\lambda}|k|!}=\frac{\mu^{-k}e^{-\mu}}{(-k)!}\,(k\le 0),
\end{equation*}
which shows that $Z$ is in a negative Poisson distribution.

\vspace{2ex}
If ${\alpha _{k}=0}$ for $k \ge 3$ in Definition \ref{def-dcp}, there is a special kind of DCP which is often used in social analysis and biology and is called Hermite distribution. 
\begin{definition}\label{herm}
    We say that $Y$ is in a Hermite distribution with parameters $(\Lambda, a_1, a_2)^{\top}$ if 
    \begin{equation*}
        Y := N_1 + 2 N_2, \qquad ~~\text{where }{N_{k}}\overset{ind.}{\sim} {\rm Poisson}(\Lambda{a _{k}})
    \end{equation*}
    with $a_1, a_2 > 0$ and $a_1 + a_2 = 1$. We denote it by $Y \sim \operatorname{Herm} (\Lambda, a_1, a_2)$.
\end{definition}
\noindent It can be seen that Hermite distribution is actually DCP with only first two active elements. In the next result, we show the sub-Gamma concentration for the sum of independent Hermite random variables. Hence the two-side Hermite distribution also enjoys sub-Gamma concentration.

\begin{theorem}
\label{thm:cdp}
Given ${Y_{i}} \sim \operatorname{DCP}\big({\alpha _{1}}(i)\lambda (i), \ldots
,{\alpha _{r}}(i)\lambda (i) \big)$ independently for $i = 1,2, \ldots ,n$, and
${\sigma _{i}^2} := \operatorname{var}{Y_{i}} = \lambda (i)\sum_{k = 1}
^{r} {{k^{2}}{\alpha _{k}}(i)} $, for non-random weights $\{w_i\}_{i=1}^n$ with $w= \mathop {\max }\limits_{1 \le i \le n}| w_i|>0$, we have
\begin{equation} \label{eq:ci1}
\pr \left\{ \Big| {\sum\limits_{i = 1}^n {w_i({Y_i} - {\rm{E}}{Y_i})} } \Big| \ge w \Big[ \big(2x\sum\limits_{i = 1}^n {\sigma _i^2} \big)^{1/2}  + rx/3 \Big] \right\}\le 2{e^{ - x}}, \qquad \forall x > 0.
\end{equation}
Specially, if $Y_i \, \overset{iid.}{\sim} \, \operatorname{Herm} (\Lambda, a_1, a_2)$, we have 
\begin{equation*}
    \pr \left\{ \big| \overline{Y} - \Lambda (a_1 + 2 a_2) \big| \geq \sqrt{\frac{2 x \Lambda (a_1 + 4 a_2)}{n}} + \frac{2 x}{3 n}\right\} \leq 2 e^{-x}, \qquad \forall x > 0.
\end{equation*}
\end{theorem}
\begin{proof}
To apply the above lemma, we need to evaluate the
log-moment-generating function of centered DCP random variables. Let
${\mu _{i}} =: \mathrm{{E}}{Y_{i}} = \lambda (i)\sum_{k = 1}
^{r} {k{\alpha _{k}}(i)}$, we have
\begin{align}\label{eq:Poissonmgf}
\log {\mathrm{{E}}} {e^{sw_i({Y_{t}} - {\mu _{i}})}} &= - sw_i{\mu _{i}} + \log
\mathrm{E} {e^{sw_i{Y_{i}}}} = - sw_i{\mu _{i}} + \log
{e^{\lambda (i)\sum _{k = 1}^{r} {{\alpha _{k}}(i)} ({e^{ksw_i}}
- 1)}}
\nonumber\\
&= \lambda (i)\sum_{k = 1}^{r}
{{\alpha _{k}}(i) \bigl({e^{ksw_i}} - ksw_i - 1 \bigr)} .
\end{align}
Hence, one can derive that
\begin{equation}\label{eq:less00}
\log {\mathrm{{E}}} {e^{sw_i({Y_{t}} - {\mu _{i}})}} \le \sum
_{k = 1}^{r} {{\alpha _{k}}(i)
\frac{\lambda (i) {{k^{2}}{w_i^{2}}s^{2}}}{{2(1 -
ksw_i)}}} \le
{\frac{ws^2\lambda (i)\sum_{k = 1}^{r}{{k^{2}}}{\alpha _{k}}(i)
}{{2(1 - rw|s|/3)}}} = \frac{{{\sigma _{i}^2}w^2{s^{2}}}}{{2(1 - rw|s|/3)}},~\left| s \right| \le \frac{3}{rw},
\end{equation}
where ${\sigma _{i}^2} = \lambda (i)\sum_{k = 1}^{r} {{k^{2}}{\alpha _{k}}(i)} $. Thus (\ref{eq:less00}) implies that $w_i({Y_{t}} - {\mu _{i}})\sim \mathrm{sub}\Gamma({{w^2}{\sigma _{i}^2}},rw/3)$.
\vspace{1ex}

From Proposition \ref{sub-GammaConcentration}(a), we obtain
$ S_n^w \sim \mathrm{sub}\Gamma({{w^2}\sum_{k = i}^n\sigma_i^2},rw/3)$. Then applying Proposition \ref{sub-GammaConcentration}(b), we get \eqref{eq:ci1}.
\end{proof}

\subsection{Proofs in Section 2}

First, we would like to introduce some equivalent definitions of sub-exponential variables. The detailed discussions and proofs can be seen in  \cite{Rigollet2019}.

\begin{proposition}[Characterizations of sub-exponential distributions]\label{prop: sub-exponential properties}
Let $X$ be a r.v. in $\R$ with $\E X = 0$. Then the following properties are equivalent (the parameters
  $K_i > 0$ are equal up to a constant
  factor)
\begin{enumerate}[\rm{(}1\rm{)}]
    \item \label{p: exponential tail}
      The tails of $X$ satisfy
      $
      \pr \{ |X| \ge t \} \le 2 e^{-t/K_1} \quad \text{for all } t \ge 0;
      $
   \item \label{p: exponential MGF}
    The MGF of $X$ satisfies
      $
      \E e^{\l X} \le e^{K_2^2 \l^2}~\text{for all}~|\l| \le \frac{1}{K_2};
      $
    \item \label{p: exponential moments}
      The moments of $X$ satisfy
      $
     (\E |X|^k)^{1/k} \le K_3 k\quad \text{for  integer}~k \ge 1;
      $
    \item \label{p: exponential MGF abs}
      The MGF of $|X|$ satisfies
      $
      \E e^{\l |X|} \le e^{K_4 \l} \quad \text{for all } 0 \le \l \le \frac{1}{K_4};
      $
    \item \label{p: exponential MGF finite}
      The MGF of $|X|$ is bounded at some point:
      $
      \E e^{|X|/K_5} \le 2.
      $
  \end{enumerate}

\end{proposition}

\begin{lemma}[Concentration for weighted sub-exponential sum]\label{sub-exponentialConcentration}
Let $\{ X_{i}\} _{i = 1}^n $ be independent $\{\operatorname{subE}(\lambda_i)\} _{i = 1}^n$ distributed with zero mean. Define $\lambda= \mathop {\max }\limits_{1 \le i \le n} \lambda_i>0$ and the non-random vector $\bm w := ({w_1}, \cdots ,{w_n})^{\top} \in {\mathbb{R}^n}$ with $w := \mathop {\max }\limits_{1 \le i \le n} |w_i|>0$, we have
\begin{enumerate}[\rm{(}a\rm{)}]
\item Closed under addition: $\sum\limits_{i = 1}^n {{w_i}{X_i}}  \sim {\rm{subE}}(\| \bm w \|_2{\lambda} )$;
\item
$\pr ( {| {\sum\limits_{i = 1}^n {{w_i}{X_i}} }| \ge t} ) \le 2e^{ { - \frac{1}{2}( {\frac{{{t^2}}}{{\left\|\bm w \right\|_2^2{\lambda ^2}}} \wedge \frac{t}{{w\lambda }}})} } = \left\{ {\begin{array}{*{20}{c}}
{2{e^{ - {t^2}/2\left\|\bm w \right\|_2^2{\lambda ^2}}},~~~~0 \le t \le \left\|\bm w \right\|_2^2\lambda /w}\\
{2{e^{ - t/2w\lambda }},~~~~~~~t > \left\|\bm w \right\|_2^2\lambda /w}
\end{array}} \right..$
\end{enumerate}
\end{lemma}
\begin{proof}
    See \cite{Rigollet2019}.
\end{proof}

\subsubsection{The proof of Lemma \ref{prop:Psub-E} and Corollary \ref{Sub-expConcentration}}

\noindent \rm{(a)}. To verified \eqref{sub-exponential tail}, using exponential Markov's inequality, we have
$$
    \pr(|X| \geq t)= \pr \left(e^{|X / \|X\|_{{\psi _1}}|} \geq e^{t / \|X\|_{{\psi _1}}}\right) \leq e^{-t / \|X\|_{{\psi _1}}|} \mathrm{E} e^{|X / \|X\|_{{\psi _1}}|} \leq 2 e^{-t / \|X\|_{{\psi _1}}}
$$
where the last inequality stems from the definition of sub-exponential norm.

\noindent \rm{(b)}. From \eqref{sub-exponential tail}, we get
\begin{align*}
\E|X{|^k} &= \int_0^\infty  \pr (|X| \ge t)k{t^{k - 1}}dt \le 2k\int_0^\infty  {{e^{ - t/\|X\|_{{\psi _1}}}}} {t^{k - 1}}dt\nonumber, \\
[\text{let}~s = {t/\|X\|_{{\psi _1}}}]~~&={{2k}}\int_0^\infty  {{e^{ - s}}} {({s}\|X\|_{{\psi _1}})^{k - 1}}\|X\|_{{\psi _1}}ds= 2{\|X\|_{{\psi _1}}^k}{k}\Gamma( {{k-1}})= 2{\|X\|_{{\psi _1}}^k}k!.
\end{align*}
\rm{(c)}. Applying Taylor's expansion to MGF, we have
\begin{align*}
\mathrm{E} \exp (s X)&=\mathrm{E}\left(1+sX+\sum_{k=2}^{\infty} \frac{(s X)^{k}}{k !}\right)=1+\sum_{k=2}^{\infty} \frac{s^{k} \mathrm{E} X^{k}}{k !}\\
[\text{Applying}~\rm{(b)}]&  \le 1 + {\rm{2}}\sum\limits_{k = 2}^\infty  {{{(s\left\| X \right\|_{{\psi _1}})}^k}}=1+\frac{2(s\left\| X \right\|_{{\psi _1}})^{2}}{1-{s\left\| X \right\|_{{\psi _1}}}},~(|{s\left\| X \right\|_{{\psi _1}}}|<1)\\
&\le 1+{4(s\left\| X \right\|_{{\psi _1}})^{2}} \leq e^{({\rm{2}} \left\| X \right\|_{{\psi _1}})^{2}s^{2}},~\text{if}~|{s}|<1/(2\left\| X \right\|_{{\psi _1}}).
\end{align*}
Therefore, $X  \sim \operatorname{subE}(2\left\| X \right\|_{{\psi _1}})$. And to prove Corollary \ref{Sub-expConcentration}, we only need to note that if $\mathrm{E} \exp(|X|/\|X\|_{\psi_1} ) \le 2$, then $X  \sim \operatorname{subE}(2\left\| X \right\|_{{\psi _1}})$ by using Lemma \ref{prop:Psub-E}(c). The result follows Lemma \ref{sub-exponentialConcentration} (b).

\subsubsection{Proof of Lemma \ref{sub-GammaMAX}}

The Jensen's inequality and ${e}^{|x|} \leq {e}^{x}+{e}^{-x}$ imply that 
\begin{align}\label{eq:sE}
\exp (s\E \max _{i=1, \ldots, n} |X_i| ) &\leq \E \operatorname{exp}(\max _{i=1, \ldots, n} |X_i| )=\E \max _{i=1, \ldots, n} e^{s |X_{i}|}\le \E \max _{i=1, \ldots, n} (e^{-s X_{i}}+e^{s X_{i}})\nonumber\\
& \le \sum_{i=1}^{n} \E (e^{-s X_{i}}+e^{s X_{i}}) \le 2n \cdot \operatorname{exp}\left(\frac{s^{2}}{2} \frac{\upsilon}{1-c|s|}\right)
\end{align}
where the last inequality is deduced by $X_i\sim \Gamma(\upsilon_{i},c_{i})$ and $-X_i\sim \Gamma(\upsilon_{i},c_{i})$.

By \eqref{eq:sE} and taking logarithm, we have
\begin{align*}
\E \max _{i=1, \ldots, n} |X_i| \leq \mathop {\inf }\limits_{|s| < {c^{ - 1}}} \frac{{\log (2n) + \frac{s^{2}}{2} \frac{\upsilon}{1-c|s|}}}{s}&=\mathop {\inf }\limits_{|s| < {c^{ - 1}}} \left\{ {\frac{{\log (2n)}}{s}{\rm{(}}1 - c|s|{\rm{) + }}\frac{s}{2}\frac{\upsilon }{{1 - c|s|}}{\rm{ + }}c\log (2n)} \right\}\\
&=\sqrt {2\upsilon \log (2n)}  + c\log (2n).
\end{align*}
This completes the proof of Lemma \ref{sub-GammaMAX}.
    
\subsection{Proofs in Section 3}

For a subset of $C\subset \R^n$, denote $C^0$ and $\overline{C}$ as the interior and closure of $C$ respectively. For a vector $x=(x_1, \ldots, x_n)^\top\in \R^n$, denote 
$\|x\|_\infty = \max_{1\le i\le n} |x_i|$ as the $\ell_\infty$-norm of $x$. For a $n\times n$ matrix $J=(J_{ij})$, let $\|J\|_\infty$ be the matrix norm induced by the $\ell_\infty$-norm on vectors in $\R^n$, i.e.,
\begin{equation}\label{equ:fanshu}
\|J\|_\infty = \max_{x\neq 0} \frac{ \|Jx\|_\infty }{\|x\|_\infty}
=\max_{1\le i\le n}\sum_{j=1}^n |J_{ij}|.
\end{equation}

Let $D$ be an open convex subset of $\R^n$.
We say a $n\times n$ function matrix $G(\mathbf{x})$ whose elements $G_{ij}(\mathbf{x})$ are functions on vectors $\mathbf{x}$,
is Lipschitz continuous on $D$ if there exists a real number of $\lambda$ such that
for~any~$\mathbf{v}\in R^n$
and any $\mathbf{x}, \mathbf{y}\in D$,
\begin{equation}\label{equ:lipschitz-lambda}
\|G(\mathbf{x})(\mathbf{v})-G(\mathbf{y})(\mathbf{v})\|_\infty
\le \lambda\|\mathbf{x}-\mathbf{y}\|_\infty \|\mathbf{v}\|_\infty,
\end{equation}
where $\lambda$ may depend on $n$ but is independent of $\mathbf{x}$ and $\mathbf{y}$.
For every fixed $n$, $\lambda$ is a constant. Given $m, M>0$, we say an $n\times n$ matrix $V=(v_{ij})$ belongs to the matrix class $\mathcal{L}_{n}(m, M)$ if
$V$ is a diagonally balanced matrix with positive elements bounded by $m$ and $M$,
\begin{equation}\label{eq1}
    v_{ii}=\sum_{j\neq i}^{n} v_{ij}, ~~i=1,\ldots, n, \quad m\le v_{ij} \le M, \quad i,j=1,\ldots,n, \quad i\neq j.
\end{equation}
 We use $V$ to denote the Jacobian matrix induced by the moment equations and show that it belongs to the matrix class
$\mathcal{L}_{n}(m, M)$. We require the inverse
of $V$, which doesn't have a closed form.
\cite{yanxu2013} proposed approximating the inverse $V^{-1}$ of $V$ by a matrix
$S=(s_{ij})$, where
\begin{equation}\label{definition:sij}
s_{ij}=\frac{\delta_{ij}}{v_{ii}}.
\end{equation}
in which $\delta_{ij}=1$ when $i=j$ and $\delta_{ij}=0$ when $i\neq j$.

\subsubsection{Preliminaries}
Before the beginning of the proof, we introduce the preliminary results that will be used in the proofs.
For a subset $C\subset \R^n$, let $C^0$ and $\overline{C}$ denote the interior and closure of $C$ in $R^n$,
respectively. Denote $\Omega(\mathbf{x}, r)$ as the open ball $\{\mathbf{y}: \|\mathbf{x}-\mathbf{y}\|< r \}$ and $\overline{\Omega(\mathbf{x}, r)}$ as its closure.
We use Newton's iterative sequence to prove the existence and consistency of the moment estimates relying on results of
\cite{1974GRAGGt}.

\begin{lemma}\label{lemma:1}
Let $F(\mathbf{x}) = (F_1(\mathbf{x}), \ldots, F_n(\mathbf{x}))^\top$ be a function vector on $\mathbf{x}\in \R^n$.
Assume that the Jacobian matrix $F^{\prime}(\mathbf{x})$ is Lipschitz continuous on an open convex set $D$ with Lipschitz constant $\lambda$.
Given $\mathbf{x}_0 \in D$, assume that $[F'(\mathbf{x}_0)]^{-1}$ exists, we have
\[
\|[F'(\mathbf{x}_0)]^{-1} \|_\infty \le \aleph, ~~ \|[F'(\mathbf{x}_0)]^{-1}F(\mathbf{x}_0)\|_\infty \le \delta,~~
h = 2 \aleph \lambda \delta \le 1,
\]
\[
\Omega(\mathbf{x}_0, t^*) \subset D^0, ~~~t^* := \frac{2}{h}(1-\sqrt{1-h})\delta = \frac{2}{1+\sqrt{1-h}}\delta \le 2\delta,
\]
where $\aleph$ and $\delta$ are positive constants that may depend on $\mathbf{x}_0$ and the dimension $n$ of $\mathbf{x}_0$.
Then the Newton iterates $\mathbf{x}_{k+1}=\mathbf{x}_k - [F'(\mathbf{x}_k)]^{-1}F(\mathbf{x}_k)$
exist and $\mathbf{x}_k \in \Omega(\mathbf{x}_0, t^*)\subset D^0$ for all $k\ge 0$;
$\mathbf{\widehat{x}} =\lim \mathbf{x}_k$ exists, $\mathbf{\widehat{x}} \in \overline{\Omega(\mathbf{x}_0, t^*)} \subset D$
and $F(\mathbf{\widehat{ x}} )=0$. Thus if $t^*\to 0$, then $\|\widehat{\bs{x}}- \bs{x}_0\|=o(1)$.
\end{lemma}

\begin{lemma}\label{lemma:2}
For a matrix $A=(a_{ij})$, define $\|A\|:=\max_{i,j} |a_{ij}|$.
If $V\in\mathcal{L}_{n}(m, M)$ at \eqref{eq1} and $n$ is large enough, we get
$$\| V^{-1}-S \| \le \frac{c_1M^2}{m^3(n-1)^2},$$
where $S$ is defined at \eqref{definition:sij} and $c_1$ is a constant that does not depend on $M$, $m$, and $n$.
\end{lemma}

\begin{lemma}\label{lemma:inverse:bound}
If $V\in\mathcal{L}_{n}(m, M)$, for $n$ which is large enough,
\[
\| V^{-1} \|_\infty \le \| V^{-1}-S \|_\infty + \|S\|_\infty \le \frac{c_1nM^2}{m^3(n-1)^2}+ \frac{1}{m}(\frac{1}{n(n-1)}+\frac{1}{n-1})
\le \frac{ c_2M^2}{ nm},
\]
where $c_2$ is a constant that does not depend on $M$, $m$, and $n$.
\end{lemma}

\subsubsection{Proof of the Theorem \ref{theorem:1}}
Then the Jacobian matrix $F^{'}({\boldsymbol{\alpha}})$ of $F({\boldsymbol{\alpha}})$ can
be calculated as follows. For $i,j=1,\dots,n,$
\begin{gather*}
    \frac{\partial{F_i}}{\partial{{\alpha_{i}}}}=-\sum_{j \neq i}^{n}e^{\alpha_{i} + \alpha_{j} + \varepsilon_{i,j}(\alpha_{i} + \alpha_{j})}\left(1+\frac{\partial \varepsilon_{i,j}(\alpha_{i} + \alpha_{j}) }{\partial \alpha_{i} }\right), \quad i=1,\dots,n, \\
    \frac{\partial{F_i}}{\partial{ {\alpha_{j}}}} = -e^{\alpha_{i} + \alpha_{j} + \varepsilon_{i,j}(\alpha_{i} + \alpha_{j})}\left(1+\frac{\partial \varepsilon_{i,j}(\alpha_{i} + \alpha_{j}) }{\partial \alpha_{j} }\right), \quad j=1,\dots,n, \quad j\neq i.
\end{gather*}
Following Assumption \ref{assumption:2}, $ -C_{0}e^{Q_{n}}\leq \varepsilon_{i,j} \leq C_{0}e^{Q_{n}}$ and $- C_{1}e^{Q_{n}}\leq \partial \varepsilon_{i,j} / \partial \alpha_{j} \leq C_{1}e^{Q_{n}}$, where $C_{0}$ and $C_{0}$ are positive constants, we have
\begin{equation*}
    -e^{Q_{n}+C_{0}e^{Q_{n}}}(1+C_{1}e^{Q_{n}})  \leq \frac{\partial{F_i}}{\partial{ {\alpha_{j}}}} \leq -e^{-Q_{n}-C_{0}e^{Q_{n}}}(1-C_{1}e^{Q_{n}}).
\end{equation*}
So for any $i\neq j$, we have the following inequality:
\begin{equation*}
    m\leq \frac{\partial{F_i}}{\partial{ {\alpha_{j}}}} \leq M.
\end{equation*}
It is not difficult to verify that $-F^{'}_{i,j}(\alpha^{*}) \in \mathcal{L}_{n}(m,M)$
 where $m=e^{-Q_{n}-C_{0}e^{Q_{n}}}(1-C_{1}e^{Q_{n}})$, $M=e^{Q_{n}+C_{0}e^{Q_{n}}}(1+C_{1}e^{Q_{n}}) $.
\vspace{2ex}

The following lemma assures that the condition holds with a large probability.
\begin{lemma}\label{lemma:4}
With probability approaching one, the following holds:
\begin{equation*}
    \max_{i=1,\dots,n} \big|\tilde{d}_{i} - \E (d_{i}) \big|= O\big( \sqrt{(n-1)\log(n-1)} + \sqrt{2\upsilon \log (2n)} + c  \log (2n) \big).
\end{equation*}
\end{lemma}
\begin{proof}
Note that $\{e_i\}_{i=1}^n$ are mutually independent and distributed in
{\em sub-gamma} distributions $\Gamma(\upsilon_{i},c_{i})(i=1,...,n)$ with respective parameters $(\upsilon_{i},c_{i})(i=1,...,n)$. Let $\max_{i=1,...,n}\upsilon_{i} = \upsilon$ and $\max_{i=1,...,n}c_{i} = c$. By Lemma \ref{sub-GammaMAX}, for each $i=1,...,n$, we have 
\begin{equation}\label{eq:sum:error}
\E (\max_{i=1,\ldots,n}|e_{i}|)  \leq \sqrt{2\upsilon \log (2n)} + c  \log (2n)
\end{equation}
Following \cite{yan2016undirectedrandom} in (C5) that they show the following inequality holds with probability approaching one:
\begin{equation}\label{eq:con3}
\max_{i=1,...,n} |d_{i} - \E(d_{i})| \le O_p \big(\sqrt {(n-1) \log (n-1)} \big),
\end{equation}
we have
\begin{equation}\label{eq:sum:error}
\begin{split}
\max_{i=1, \ldots, n} |\tilde{d}_{i} - \E ({d}_{i}) | & \le \max_i|d_{i} - \E (d_{i}) | + \max_i |e_{i}| \\
&= O_p \big( \sqrt{(n-1)\log(n-1)} +  \sqrt{2\upsilon \log (2n)} + c  \log (2n) \big),
\end{split}
\end{equation}
and it is what we need to prove.
\end{proof}

Now, we present the proof of Theorem \ref{theorem:1}.
\begin{proof}
Let
\begin{equation*}
    g_{ij}(\alpha) = \left(\frac{\partial^{2} F_{i}}{\partial \alpha_{1}\partial \alpha_{j}},....,\frac{\partial^{2} F_{i}}{\partial \alpha_{n}\partial \alpha_{j}} \right)^{\top}.
\end{equation*}
It is easy to verify that
\begin{equation*}
    \frac{\partial^{2} F_i}{\partial{\alpha_{i}}^{2}}=\sum_{j \neq i}^{n}e^{\alpha_{i} + \alpha_{j} + \varepsilon_{i,j}(\alpha_{i},\alpha_{j})}\left[\Big(1+\frac{\partial \varepsilon_{i,j}(\alpha_{i},\alpha_{j}) }{\partial \alpha_{i} }\Big)^{2} + \frac{\partial^{2} \varepsilon_{i,j}(\alpha_{i},\alpha_{j}) }{\partial {\alpha_{i}}^{2} }\right],\quad i=1,\dots,n,
\end{equation*}
and
\begin{equation*}
    \begin{aligned}
        & \frac{\partial^{2} F_i}{\partial \alpha_{j} \partial \alpha_{i}} \\
        & =e^{\alpha_{i} + \alpha_{j} + \varepsilon_{i,j}(\alpha_{i} ,\alpha_{j})}\left[\Big(1+\frac{\partial \varepsilon_{i,j}(\alpha_{i} ,\alpha_{j}) }{\partial \alpha_{j} }\Big) \Big(1+\frac{\partial \varepsilon_{i,j}(\alpha_{i},\alpha_{j}) }{\partial \alpha_{i} }\Big)+ \frac{\partial^{2} \varepsilon_{i,j}(\alpha_{i}, \alpha_{j}) }{\partial \alpha_{j} \partial \alpha_{i} }\right], \quad j=1,\dots,n, \quad j\neq i.
    \end{aligned}
\end{equation*}
Following Assumption \ref{assumption:2} that $ -C_{2}e^{Q_{n}}\leq \frac{\partial^{2} \varepsilon_{i,j}(\alpha_{i},\alpha_{j}) }{\partial \alpha_{j} \partial \alpha_{i} } \leq C_{2}e^{Q_{n}}$ where $C_{2}$ is a positive constant, we have
\begin{equation}\label{eq:app1}
\begin{split}
     \left|\frac{\partial^{2} F_i}{\partial \alpha_{j} \partial \alpha_{i}} \right| \leq e^{Q_{n}+C_{0}e^{Q_{n}}} \big(1+2C_{1}e^{Q_{n}}+C_{1}^{2}e^{2Q_{n}}+C_{2}e^{Q_{n}} \big).
\end{split}
\end{equation}
Let $M_{1} =e^{Q_{n}+C_{0}e^{Q_{n}}}(1+2C_{1}e^{Q_{n}}+C_{1}^{2}e^{2Q_{n}}+C_{2}e^{Q_{n}}) $. This leads to $\|g_{ii}(\alpha)\|_{1} \leq 2(n-1)M_{1}$, where $\|x\|_{1} = \sum_{i}|x_{i}|$ for a general vector $x$. On the other hand, we note that when $i \neq j$ and $k \neq i,j$, there exists
\begin{equation*}
    \frac{\partial^{2} F_i}{\partial \alpha_{k} \partial \alpha_{j}}=0
\end{equation*}
which leads to that $\|g_{ij}(\alpha)\|_{1} \leq 2M_{1}$ when $i \neq j$. Consequently, for any vector $v$,
\begin{equation*}
\begin{split}
\max_{i} \sum_{j}\left[\frac{\partial F_{i}}{\partial \alpha_{j}}(x)- \frac{\partial F_{i}}{\partial \alpha_{j}}(y) \right]v_{j} & \leq \|v\|_{\infty} \max_{i}\sum_{j} \left|\frac{\partial F_{i}}{\partial \alpha_{j}}(x)- \frac{\partial F_{i}}{\partial \alpha_{j}}(y) \right| \\
& = \|v\|_{\infty} \max_{i} \sum_{j} \left|\int_{0}^{1}g_{i}\big(tx+(1-t)y \big) \, dt \right|\\
& = \|v\|_{\infty} \|x-y\|_{\infty}\max_{i}\sum_{j}\max_{\alpha \in D }\|g_{ij}(\alpha)\|_{1}   \\
&\leq 4M_{1}(n-1) \|v\|_{\infty} \|x-y\|_{\infty}
\end{split}
\end{equation*}
It shows that $F^{'}(x)$ is Lipschitz continuous with the lipschitz coefficient $\lambda = 4M_{1}(n-1)$. For any $\alpha \in D$, we can define the Newton's iterative sequence with the starting point $\alpha^{(0)} :=  \alpha$, i.e.,
\begin{equation*}
    \alpha^{(k+1)} = \alpha^{(k)} - [F'(\alpha^{(k)})]^{-1}F(\alpha^{(k)}), \qquad k = 0,1,...
\end{equation*}
which shows that $F'(\alpha) \in \mathcal{L}_{n}(m, M)$ with
\begin{equation*}
    m=e^{-Q_{n}-C_{0}e^{-Q_{n}}}(1-C_{1}e^{-Q_{n}}), \qquad M=e^{Q_{n}+C_{0}e^{Q_{n}}}(1+C_{1}e^{Q_{n}}) .
\end{equation*}
By lemma \ref{lemma:2} and $M^{2}/m^{3} = o(n)$, we have 
\begin{eqnarray*}
\big\| [F'(\bs{\alpha})]^{-1}F(\bs{\alpha}) \big\|_\infty
& \le & \big\| [F'(\bs{\alpha})]^{-1} \big\|_\infty \big\| F(\bs{\alpha}) \big\|_\infty  \\
& \le & \left[\frac{c_1nM^2}{m^{3}(n-1)^2}+ \frac{1}{m(n-1)} \right]\| F(\bs{\alpha})\|_\infty \\
& \le &  O_{p} \left(\frac{ c_2M^2}{nm^{3}} \big(\sqrt{(n-1)\log(n-1)}+\sqrt{2\upsilon \log (2n)} + c  \log (2n) \big) \right)\\
& \le &  O_{p} \left(\frac{1}{n}e^{14Q_{n}+e^{Q_{n}}} \big(\sqrt{(n-1)\log(n-1)}+\sqrt{2\upsilon \log (2n)} + c  \log (2n) \big)\right).
\end{eqnarray*}
Hence
\begin{equation*}
    \aleph = O\left(\frac{1}{n}e^{4Q_{n}+e^{Q_{n}}}\right), \quad \delta = O\left(\frac{1}{n}e^{14Q_{n}+e^{Q_{n}}} \big(\sqrt{(n-1)\log(n-1)}+\sqrt{2\upsilon \log (2n)} + c\log (2n) \big)\right),
\end{equation*}
and thus
\begin{equation*}
    \begin{aligned}
        h = 2 \aleph \lambda \delta = & O\left(\frac{1}{n}e^{5Q_{n}+e^{Q_{n}}}\right)\times (n-1)O(e^{4Q_{n}+e^{Q_{n}}}) \\
        &\times O\left(\frac{1}{n}e^{14Q_{n}+e^{Q_{n}}} \big(\sqrt{(n-1)\log(n-1)}+\sqrt{2\upsilon \log (2n)} + c  \log (2n) \big) \right).
    \end{aligned}
\end{equation*}
If $e^{Q_{n}+e^{Q_{n}}} = \left(\frac{(\sqrt{(n-1)\log(n-1)}+\sqrt{2\upsilon \log (2n)} + c  \log (2n))}{n}\right)^{-1/23}$ and $n \rightarrow \infty$, we have $h=o(1)$. This verifies the conditions in Lemma \ref{lemma:1}. Therefore, $\lim_{k\rightarrow \infty }\alpha^{(k)}$ exists, and it is exactly $\widehat{\alpha}$. By lemma \ref{lemma:1}, it satisfies
\begin{equation*}
    \| \widehat{\alpha} - \alpha^{*}\|_{\infty} \leq 2\delta = O_{p}\left(\frac{1}{n}e^{16Q_{n}+e^{Q_{n}}} \big(\sqrt{(n-1)\log(n-1)}+\sqrt{2\upsilon \log (2n)} + c  \log (2n) \big)\right).
\end{equation*}
This is the consistency we need to prove.
\end{proof}

\subsubsection{Proof of the Theorem 2}
\noindent To prove Theorem \ref{theorem:2}, we should introduce the following lemma and proposition.
\begin{lemma}\label{lemma:asy}
If $V \in \mathcal{L}_n(m, M)$, $W=V^{-1}-S$ and $U=\operatorname{cov}\big[W\{\mathbf{{d}}- \E(\mathbf{d})\}\big]$, then
\begin{equation*}
\|U\|\leq \|V^{-1} - S\| + \frac{2M}{m^{2}(n-1)^{2}}.
\end{equation*}
\end{lemma}
\begin{proof}
The proof of lemma \ref{lemma:asy} is similar to that of Proposition 1 in \cite{2015Asymptotic}, so we omit it. 
\end{proof}

\begin{proposition}\label{pro:1}
Assume that \\
(C1) $V:= \var(\mathbf{\tilde{d}}) \in \mathcal{L}_n(m, M)$; \\
(C2) $({\tilde{d}_{i} - \E(d_{i})})/{v_{ii}^{1/2}} $are asymptotically standard normal as $n\to \infty $. \\
If $M/m^2=o(n)$, then
for any fixed $k$, the first $k$ elements of $S(\mathbf{\tilde{d}}-\E (\mathbf{{d}}))$ are asymptotically normal distribution
with mean zero and the covariance is given by the upper $k\times k$ submatrix of the diagonal matrix
$B=diag(1/v_{11}, \ldots, 1/v_{nn})$, where $S$ is the approximate inverse of $V$ defined at \eqref{definition:sij}.
\end{proposition}
In Proposition 1 in \cite{yan2016undirectedrandom}, they show that if $M/m^{2}=o(n)$, then the vector $\big(d_{1} - \E(d_{1}),...,d_{r} - \E(d_{r}) \big)^{\top}$ is asymptotically normally distributed with mean zero and covariance matrix diag$(v_{11},...,v_{rr})$ for a fixed $r \geq 1$. Note that random variables $\{ e_i\}_{i=1}^n$ are mutually independent and distributed by
sub-gamma distributions with respective parameters $(\upsilon_{i},c_{i})(i=1,...,n)$. Recall that $\max_{i=1,...,n}\upsilon_{i} = \upsilon$ and $\max_{i=1,...,n}c_{i} = c$ for any $\tau > 0$, by Chebyshev's inequality, we have
\begin{equation*}
    \pr \left(|\frac{e_{i} }{v_{ii}}| > \tau \right) = \pr \left(|e_{i} | > \tau  v_{ii} \right)\leq\frac{\var (e_{i} )}{\tau^{2} (v_{ii})^{2}}
\end{equation*}
According to \cite{boucheron2013concentration}, we have $\E X^{2}\leq 8(\upsilon+c^{2})$. Then $\frac{\var (e_{i} )}{\tau^{2} (v_{ii})^{2}}   \leq \frac{8(\upsilon+c^{2}) )}{\tau^{2} (v_{ii})^{2}}$ holds.
If $M/m^{2}=o(n)$, we get
\begin{equation*}
    (v_{ii})^{1/2}\big[S({\tilde{d} - \E({d})})\big]_{i}=\frac{{d}_{i} - \E({d}_{i})}{(v_{ii})^{1/2}}+\frac{e_{i} }{v_{ii}}=\frac{{d}_{i}-\E({d}_{i})}{v_{ii}^{1/2}}+o_{p}(1)
\end{equation*}
 Therefore, for any fixed $k$, $({\tilde{d}_{i} - \E({d}_{i})})/{(v_{ii})^{1/2}}, i=1,...,k,$ are asymptotically independent and standard normal distributions.

\begin{proof}[Proof of Theorem \ref{theorem:2}]

Let $\widehat{\gamma}_{ij}=\widehat{\alpha}_i + \widehat{\alpha}_j -\alpha_{i}^{*}- \alpha_{j}^{*}$
and assume
\begin{equation}\label{equ:theorem:part2}
    \max_{i\neq j} |\hat{\gamma}_{ij}| = O\left(\frac{1}{n}e^{14Q_{n}+e^{Q_{n}}} \big(\sqrt{(n-1)\log(n-1)}+\sqrt{2\upsilon \log (2n)}+c\log (2n) \big) \right).
\end{equation}
For $i=1,\ldots,n$, by Taylor's expansion, we have
\begin{eqnarray*}
\tilde{d}_{i} - \E({d}_{i}) & = & \sum_{j\neq i} (e^{\widehat{\alpha}_{i}+\widehat{\alpha}_{j}+\varepsilon_{i,j}(\hat{\alpha}_{i},\hat{\alpha}_{j})}  - e^{\alpha_{i}^{*}+\alpha_{j}^{*}+\varepsilon_{i,j}(\alpha_{i}^{*},\alpha_{j}^{*})} ) \\
& = & \sum_{j\neq i} \Big[(e^{\alpha_{i}^{*}+\alpha_{j}^{*}+\varepsilon_{i,j}(\alpha_{i}^{*},\alpha_{j}^{*})})'(e^{\widehat{\alpha}_{i}+\widehat{\alpha}_{j}+\varepsilon_{i,j}(\widehat{\alpha}_{i},\widehat{\alpha}_{j})}  - e^{\alpha_{i}^{*}+\alpha_{j}^{*}+\varepsilon_{i,j}(\alpha_{i}^{*},\alpha_{j}^{*})} )\Big]
+ h_i,
\end{eqnarray*}
where $h_i=\frac{1}{2}\sum_{j\neq i} \left.(e^{\alpha_{i} + \alpha_{j} + \varepsilon_{i,j}(\alpha_{i} ,\alpha_{j})})''\right|_{\alpha_{i} + \alpha_j = \theta_{ij}}[(\widehat{\alpha}_{i} + \widehat{\alpha}_{j}) - ( \alpha_{i}^{*} + \alpha_{j}^{*}) ]^2$
and ${\theta}_{ij} = t_{ij}( \alpha_{i}^{*}+ \alpha_{j}^{*})
+ (1-t_{ij})(\widehat{\alpha}_i + \widehat{\alpha}_j)$, $0<t_{ij}<1$.
\vspace{2ex}

Writing the above expressions in matrices, we have
\begin{eqnarray*}
\mathbf{\tilde{d}} - \E \mathbf{{d}} &=& V(\bs{\widehat{\alpha}} - \bs{\alpha}) + \mathbf{h}.
\end{eqnarray*}
Equivalently,
\begin{eqnarray*}
\bs{\widehat{\alpha}} - \bs{\alpha} & = & V^{-1}(\mathbf{\tilde{d}} - \E \mathbf{d}) + V^{-1} \mathbf{h} \\
& = & S(\mathbf{\tilde{d}} - \E \mathbf{{d}}) + W( \mathbf{\tilde{d}} - \E \mathbf{d}) + V^{-1} \mathbf{h},
\end{eqnarray*}
where $\mathbf{h}=(h_1, \ldots, h_n)^{\top}$. Now that $\mu''({\theta}_{ij})=O(e^{4Q_{n}})$, then we get
\[
|h_i|\le \frac{1}{2}(n-1)e^{4Q_{n}}\widehat{\gamma}_{ij}^2,
\]
Therefore,
\begin{eqnarray*}
|(V^{-1} \mathbf{h})_i| & = & |(S\mathbf{h})_i| + |(W\mathbf{h})_i|
\\
& \le & \max_i \frac{|h_i|}{v_{ii}} + \|W\|\sum_i|h_i|
\\
& \le & O\left(\frac{3e^{4Q_{n}}\widehat{\gamma}_{ij}^2}{2m} +  \frac{ c_1M^2}{m^3(n-1)^2}\times \frac{1}{2}n(n-1)e^{4Q_{n}} \widehat{\gamma}_{ij}^2\right)
\\
&\le & O\left( \frac{m^2+C_3M^2}{2m^3}e^{4Q_{n}} \widehat{\gamma}_{ij}^2\right)
\\
& = & O \left(\frac{1}{n}e^{41Q_{n}+e^{Q_{n}}}\big(\sqrt{(n-1)\log(n-1)}+\sqrt{2\upsilon \log (2n)} + c  \log (2n)\big)^2 \right).
\end{eqnarray*}
If $\frac{1}{n}e^{41Q_{n}+e^{Q_{n}}}\big(\sqrt{(n-1)\log(n-1)}+\sqrt{2\upsilon \log (2n)} + c\log (2n)\big)^2=o(n^{1/2})$,
then, $(V^{-1} \mathbf{h})_i = o(n^{-1/2}),$
by lemma \ref{lemma:asy}, we have
\begin{equation*}
\begin{split}
&\var[W\{d_{i}-E(d_{i})\}+W\{e_{i}\}] \\
& =  U_{ii}+2\operatorname{cov}([W\{d_{i} - \E(d_{i})]_{i},W\{e_{i}]_{i}) + \var(\Sigma_{j}W_{ij}e_{j}) \\
& \le  O\left(\frac{e^{6Q_{n}}}{(n-1)^{2}}\right) + 2\Sigma_{j}w^{2}_{ij} \operatorname{cov} (d_{j} - \E(d_{j}),e_{j})+8(\upsilon+c^{2})n\|W\|^{2}\\
& \le  O \left(\frac{e^{6Q_{n}}}{(n-1)^{2}} + \frac{e^{6Q_{n}}8(\upsilon+c^{2})}{(n-1)^{3}} \right).
\end{split}
\end{equation*}
If $8e^{6Q_{n}}(\upsilon+c^{2}) = o(n^{1/2})$, by Chebyshev's inequality, we obtain that
\begin{equation*}
    \pr \left( \frac{\big[W\{\tilde{d} - \E({d})\}\big]_{i}}{n^{-1/2}}>\epsilon \right) \leq \frac{n \var [W\{\tilde{d} - \E({d})\}]_{i}}{\epsilon^{2}} =  o(n^{-1/2})
\end{equation*}
For arbitrarily given $\epsilon >0$, it shows that
\begin{equation}\label{equ:W}
    \big[W\{\tilde{d} - \E({d})\} \big]_{i} = o(n^{-1/2})
\end{equation}
By the first part of this theorem, \eqref{equ:theorem:part2} holds with probability approaching $1$.
Consequently, by \eqref{equ:W}, we have
\[
(\bs{\widehat{\alpha}} - \bs{\alpha}^*)_i = \big[S(\mathbf{\tilde{d}} - \E (\mathbf{{d}}) ) \big]_i + o_p(n^{-1/2}).
\]
Therefore, Theorem \ref{theorem:2} follows after Proposition \ref{pro:1}.
\end{proof}

\bibliographystyle{plainnat}
\bibliography{reference}

\begin{thebibliography}{43}
\providecommand{\natexlab}[1]{#1}
\providecommand{\url}[1]{\texttt{#1}}
\expandafter\ifx\csname urlstyle\endcsname\relax
  \providecommand{\doi}[1]{doi: #1}\else
  \providecommand{\doi}{doi: \begingroup \urlstyle{rm}\Url}\fi

\bibitem[Albert and Barab{\'a}si(2002)]{albert2002}
R{\'e}ka Albert and Albert-L{\'a}szl{\'o} Barab{\'a}si.
\newblock Statistical mechanics of complex networks.
\newblock \emph{Reviews of modern physics}, 74\penalty0 (1):\penalty0 47, 2002.

\bibitem[Bickel et~al.(2011)Bickel, Chen, Levina,
  et~al.]{degree3bickel2011method}
Peter~J Bickel, Aiyou Chen, Elizaveta Levina, et~al.
\newblock The method of moments and degree distributions for network models.
\newblock \emph{The Annals of Statistics}, 39\penalty0 (5):\penalty0
  2280--2301, 2011.

\bibitem[Blitzstein and Diaconis(2011)]{degree2blitzstein2011sequential}
Joseph Blitzstein and Persi Diaconis.
\newblock A sequential importance sampling algorithm for generating random
  graphs with prescribed degrees.
\newblock \emph{Internet Mathematics}, 6\penalty0 (4):\penalty0 489--522, 2011.

\bibitem[Boucheron et~al.(2013)Boucheron, Lugosi, and
  Massart]{boucheron2013concentration}
St{\'e}phane Boucheron, G{\'a}bor Lugosi, and Pascal Massart.
\newblock \emph{Concentration inequalities: A nonasymptotic theory of
  independence}.
\newblock Oxford university press, 2013.

\bibitem[Britton et~al.(2006)Britton, Deijfen, and
  Martin-L{\"o}f]{degree1britton2006generating}
Tom Britton, Maria Deijfen, and Anders Martin-L{\"o}f.
\newblock Generating simple random graphs with prescribed degree distribution.
\newblock \emph{Journal of Statistical Physics}, 124\penalty0 (6):\penalty0
  1377--1397, 2006.

\bibitem[Chatterjee et~al.(2011)Chatterjee, Diaconis, and
  Sly]{chatterjee2011random}
Sourav Chatterjee, Persi Diaconis, and Allan Sly.
\newblock Random graphs with a given degree sequence.
\newblock \emph{The Annals of Applied Probability}, pages 1400--1435, 2011.

\bibitem[Cutillo et~al.(2010)Cutillo, Molva, and Strufe]{Cutillo2010Privacy}
Leudo~Antonio Cutillo, Refik Molva, and Thorsten Strufe.
\newblock Privacy preserving social networking through decentralization.
\newblock In \emph{International Conference on Wireless On-demand Network
  Systems and Services}, 2010.

\bibitem[Dwork et~al.(2006)Dwork, McSherry, Nissim, and Smith]{Dwork2006}
Cynthia Dwork, Frank McSherry, Kobbi Nissim, and Adam Smith.
\newblock Calibrating noise to sensitivity in private data analysis.
\newblock In \emph{Theory of cryptography conference}, pages 265--284.
  Springer, 2006.

\bibitem[Erdos et~al.(1960)Erdos, R{\'e}nyi, et~al.]{Erd1959On}
Paul Erdos, Alfr{\'e}d R{\'e}nyi, et~al.
\newblock On the evolution of random graphs.
\newblock \emph{Publ. Math. Inst. Hung. Acad. Sci}, 5\penalty0 (1):\penalty0
  17--60, 1960.

\bibitem[Fan and Lu(2002)]{Fan2002Connected}
Chung Fan and Linyuan Lu.
\newblock Connected components in random graphs with given expected degree
  sequences.
\newblock \emph{Annals of Combinatorics}, 6\penalty0 (2):\penalty0 125--145,
  2002.

\bibitem[Fan et~al.(2020)Fan, Zhang, and Yan]{fan2020}
Yifan Fan, Huiming Zhang, and Ting Yan.
\newblock Asymptotic theory for differentially private generalized
  $\beta$-models with parameters increasing.
\newblock \emph{Statistics and Its Interface}, 13\penalty0 (3):\penalty0
  385--398, 2020.

\bibitem[Fienberg(2012)]{fienberg2012brief}
Stephen~E Fienberg.
\newblock A brief history of statistical models for network analysis and open
  challenges.
\newblock \emph{Journal of Computational and Graphical Statistics}, 21\penalty0
  (4):\penalty0 825--839, 2012.

\bibitem[Giné and Nickl(2015)]{Gin2015Mathematical}
Evarist Giné and Richard Nickl.
\newblock Mathematical foundations of infinite-dimensional statistical models.
\newblock 2015.
\newblock \doi{doi:10.1017/CBO9781107337862}.

\bibitem[Gragg and Tapia(1974)]{1974GRAGGt}
W.B. Gragg and R.A. Tapia.
\newblock Optimal error bounds for the newton-kantorovich theorem.
\newblock \emph{SIAM Journal on Numerical Analysis}, 11\penalty0 (1):\penalty0
  10--13, 1974.

\bibitem[Hay~M. and D.(2009)]{Hay2009}
Miklau~G. Hay~M., Li~C. and Jensen D.
\newblock Accurate estimation of the degree distribution of private networks.
\newblock In \emph{Ninth IEEE International Conference on Data Mining}, pages
  169--178. IEEE, 2009.

\bibitem[Hillar and Wibisono(2013)]{hillar2013maximum}
Christopher Hillar and Andre Wibisono.
\newblock Maximum entropy distributions on graphs.
\newblock \emph{\em Avaible at: \url{http://arxiv.org/abs/1301.3321}}, 2013.

\bibitem[Holland and Leinhardt(1981)]{Holland1981An}
Paul~W. Holland and Samuel Leinhardt.
\newblock An exponential family of probability distributions for directed
  graphs.
\newblock \emph{Journal of the American Statistical Association}, 76\penalty0
  (373):\penalty0 33--50, 1981.

\bibitem[Inusah and Kozubowski(2006)]{Inusah2006A}
Seidu Inusah and Tomasz~J. Kozubowski.
\newblock A discrete analogue of the laplace distribution.
\newblock \emph{Journal of Statal Planning and Inference}, 136\penalty0
  (3):\penalty0 1090--1102, 2006.

\bibitem[Karwa and Slavkovi{\'c}(2016)]{Karwa2016}
Vishesh Karwa and Aleksandra Slavkovi{\'c}.
\newblock Inference using noisy degrees: Differentially private $\beta $-model
  and synthetic graphs.
\newblock \emph{The Annals of Statistics}, 44\penalty0 (1):\penalty0 87--112,
  2016.

\bibitem[Lu and Miklau(2014)]{Lu2014Exponential}
Wentian Lu and Gerome Miklau.
\newblock Exponential random graph estimation under differential privacy.
\newblock In \emph{In proceedings of the 20th ACM SIGKDD international
  conference on Knowlege discovery and data mining}, 2014.

\bibitem[Luo and Qin(2022{\natexlab{a}})]{luo2021Ordered}
Jing Luo and Hong Qin.
\newblock Asymptotic in the ordered networks with a noisy degree sequence.
\newblock \emph{Journal of Systems Science and Complexity}, 35\penalty0
  (3):\penalty0 1137--1153, 2022{\natexlab{a}}.

\bibitem[Luo and Qin(2022{\natexlab{b}})]{luo2022asymptotic}
Jing Luo and Hong Qin.
\newblock Asymptotic in a class of network models with a difference private
  degree sequence.
\newblock \emph{Statistics and Its Interface}, 15\penalty0 (3):\penalty0
  383--397, 2022{\natexlab{b}}.

\bibitem[Luo et~al.(2020)Luo, Qin, and Wang]{luo2020asymptotic}
Jing Luo, Hong Qin, and Zhenghong Wang.
\newblock Asymptotic distribution in directed finite weighted random graphs
  with an increasing bi-degree sequence.
\newblock \emph{Acta Mathematica Scientia}, 40\penalty0 (2):\penalty0 355--368,
  2020.

\bibitem[Luo et~al.(2022)Luo, Liu, and Wang]{luo2021affiliation}
Jing Luo, Tour Liu, and Qiuping Wang.
\newblock Affiliation weighted networks with a differentially private degree
  sequence.
\newblock \emph{Statistical Papers}, 63:\penalty0 383--397, 2022.

\bibitem[Mccullagh and Nelder(1989)]{Mccullagh1989Generalized}
P.~Mccullagh and J.~A. Nelder.
\newblock \emph{Generalized Linear Models}.
\newblock Chapman and Hall, 1989.

\bibitem[Mosler(2017)]{Mosler2015}
Karl Mosler.
\newblock Ernesto estrada and philip a. knight (2015): A first course in
  network theory, oxford university press, 272 pp., {\textsterling}29.99, isbn
  9780198726463.
\newblock \emph{Statistical Papers}, 58\penalty0 (4):\penalty0 1283--1284, Dec
  2017.
\newblock ISSN 1613-9798.
\newblock \doi{10.1007/s00362-017-0961-1}.
\newblock URL \url{https://doi.org/10.1007/s00362-017-0961-1}.

\bibitem[Pr{\'e}kopa(1952)]{prekopa1952composed}
Andr{\'a}s Pr{\'e}kopa.
\newblock On composed poisson distributions, iv.
\newblock \emph{Acta Mathematica Academiae Scientiarum Hungarica}, 3\penalty0
  (4):\penalty0 317--325, 1952.

\bibitem[Rigollet and H{\"u}tter(2019)]{Rigollet2019}
Philippe Rigollet and Jan-Christian H{\"u}tter.
\newblock High dimensional statistics.
\newblock 2019.
\newblock URL \url{http://www-math.mit.edu/~rigollet/PDFs/RigNotes17.pdf}.

\bibitem[Rinaldo et~al.(2013)Rinaldo, Petrovi{\'c}, Fienberg,
  et~al.]{rinaldo2013maximum}
Alessandro Rinaldo, Sonja Petrovi{\'c}, Stephen~E Fienberg, et~al.
\newblock Maximum lilkelihood estimation in the $beta$-model.
\newblock \emph{The Annals of Statistics}, 41\penalty0 (3):\penalty0
  1085--1110, 2013.

\bibitem[Steutel and Van~Harn(2003)]{steutel2003infinite}
Fred~W Steutel and Klaas Van~Harn.
\newblock \emph{Infinite divisibility of probability distributions on the real
  line}.
\newblock CRC Press, 2003.

\bibitem[Vershynin(2018)]{Vershynin2018}
R.~Vershynin.
\newblock High-dimensional probability: An introduction with applications in
  data science.
\newblock 2018.

\bibitem[Yan and Xu(2013)]{yanxu2013}
Ting Yan and Jinfeng Xu.
\newblock A central limit theorem in the $\beta$-model for undirected random
  graphs with a diverging number of vertices.
\newblock \emph{Biometrika}, 100\penalty0 (2):\penalty0 519--524, 2013.

\bibitem[Yan et~al.(2015)Yan, Zhao, and Qin]{2015Asymptotic}
Ting Yan, Yunpeng Zhao, and Hong Qin.
\newblock Asymptotic normality in the maximum entropy models on graphs with an
  increasing number of parameters.
\newblock \emph{Journal of Multivariate Analysis}, 133:\penalty0 61--76, 2015.

\bibitem[Yan et~al.(2016)Yan, Qin, and Wang]{yan2016undirectedrandom}
Ting Yan, Hong Qin, and Hansheng Wang.
\newblock Asymptotics in undirected random graph models parameterized by the
  strengths of vertices.
\newblock \emph{Statistica Sinica}, 26:\penalty0 273--293, 2016.

\bibitem[Yan et~al.(2019)Yan, Jiang, Fienberg, and Leng]{Yan2019}
Ting Yan, Binyan Jiang, Stephen~E. Fienberg, and Chenlei Leng.
\newblock Statistical inference in a directed network model with covariates.
\newblock \emph{Journal of the American Statistical Association}, 114\penalty0
  (526):\penalty0 857--868, 2019.

\bibitem[Yuan et~al.(2011)Yuan, Lei, and Yu]{Yuan2011Personalized}
Mingxuan Yuan, Chen Lei, and Philip~S. Yu.
\newblock Personalized privacy protection in social networks.
\newblock \emph{Proceedings of the Vldb Endowment}, 4\penalty0 (2):\penalty0
  141--150, 2011.

\bibitem[Zhang and Chen(2021)]{zhang2020concentration}
Huiming Zhang and Song~Xi Chen.
\newblock Concentration inequalities for statistical inference.
\newblock \emph{Communications in Mathematical Research}, 37\penalty0
  (1):\penalty0 1--85, 2021.

\bibitem[{Zhang} and {Jia}(2022)]{zhang2022elastic}
Huiming {Zhang} and Jinzhu {Jia}.
\newblock Elastic-net regularized high-dimensional negative binomial
  regression: Consistency and weak signals detection.
\newblock \emph{Statistica Sinica}, 32:\penalty0 181--207, 2022.

\bibitem[Zhang and Li(2016)]{zhang2016characterizations}
Huiming Zhang and Bo~Li.
\newblock Characterizations of discrete compound poisson distributions.
\newblock \emph{Communications in Statistics-Theory and Methods}, 45\penalty0
  (22):\penalty0 6789--6802, 2016.

\bibitem[Zhang and Wei(2022)]{zhang2021sharper}
Huiming Zhang and Haoyu Wei.
\newblock Sharper sub-weibull concentrations.
\newblock \emph{Mathematics}, 10\penalty0 (13):\penalty0 2252, 2022.

\bibitem[Zhang et~al.(2014)Zhang, Liu, and Li]{zhang2014notes}
Huiming Zhang, Yunxiao Liu, and Bo~Li.
\newblock Notes on discrete compound poisson model with applications to risk
  theory.
\newblock \emph{Insurance: Mathematics and Economics}, 59:\penalty0 325--336,
  2014.

\bibitem[Zhao et~al.(2012)Zhao, Levina, and Zhu]{degree4zhao2012consistency}
Yunpeng Zhao, Elizaveta Levina, and Ji~Zhu.
\newblock Consistency of community detection in networks under degree-corrected
  stochastic block models.
\newblock \emph{The Annals of Statistics}, 40\penalty0 (4):\penalty0
  2266--2292, 2012.

\bibitem[Zhou et~al.(2008)Zhou, Pei, and Luk]{Zhou2008A}
Bin Zhou, Jian Pei, and Wo~Shun Luk.
\newblock A brief survey on anonymization techniques for privacy preserving
  publishing of social network data.
\newblock \emph{Acm Sigkdd Explorations Newsletter}, 10\penalty0 (2):\penalty0
  12--22, 2008.

\end{thebibliography}
\newpage

\begin{figure}[H]
\centering
\includegraphics[width=0.90\textwidth]{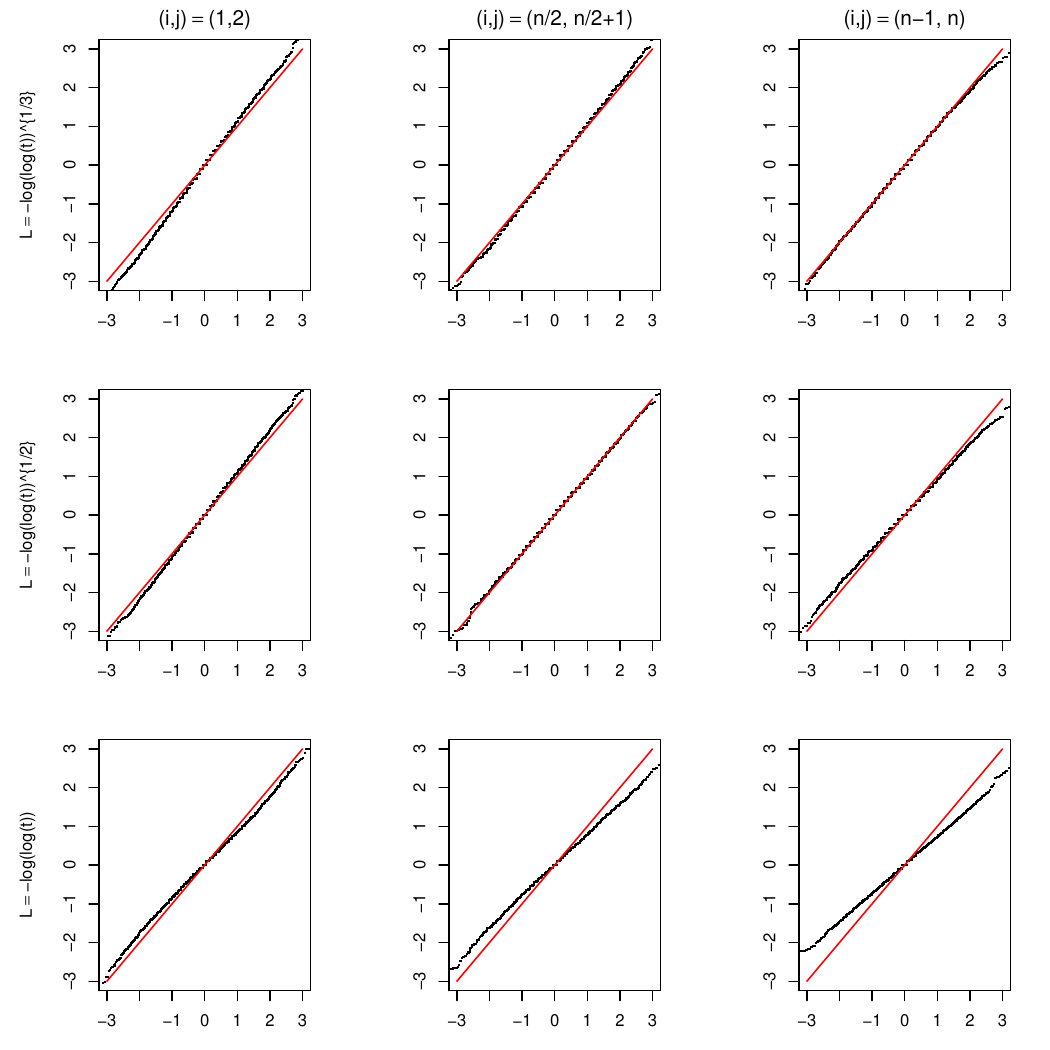}
\caption{In case of model $M_{log}$, the QQ plots of $\xi_{ij}$ with red color for $\widehat{\xi}_{ij}$ ($n=100~~ and~~ \epsilon = 2$). }
\label{Fig1:log}
\end{figure}

\begin{figure}[H]
\centering
\includegraphics{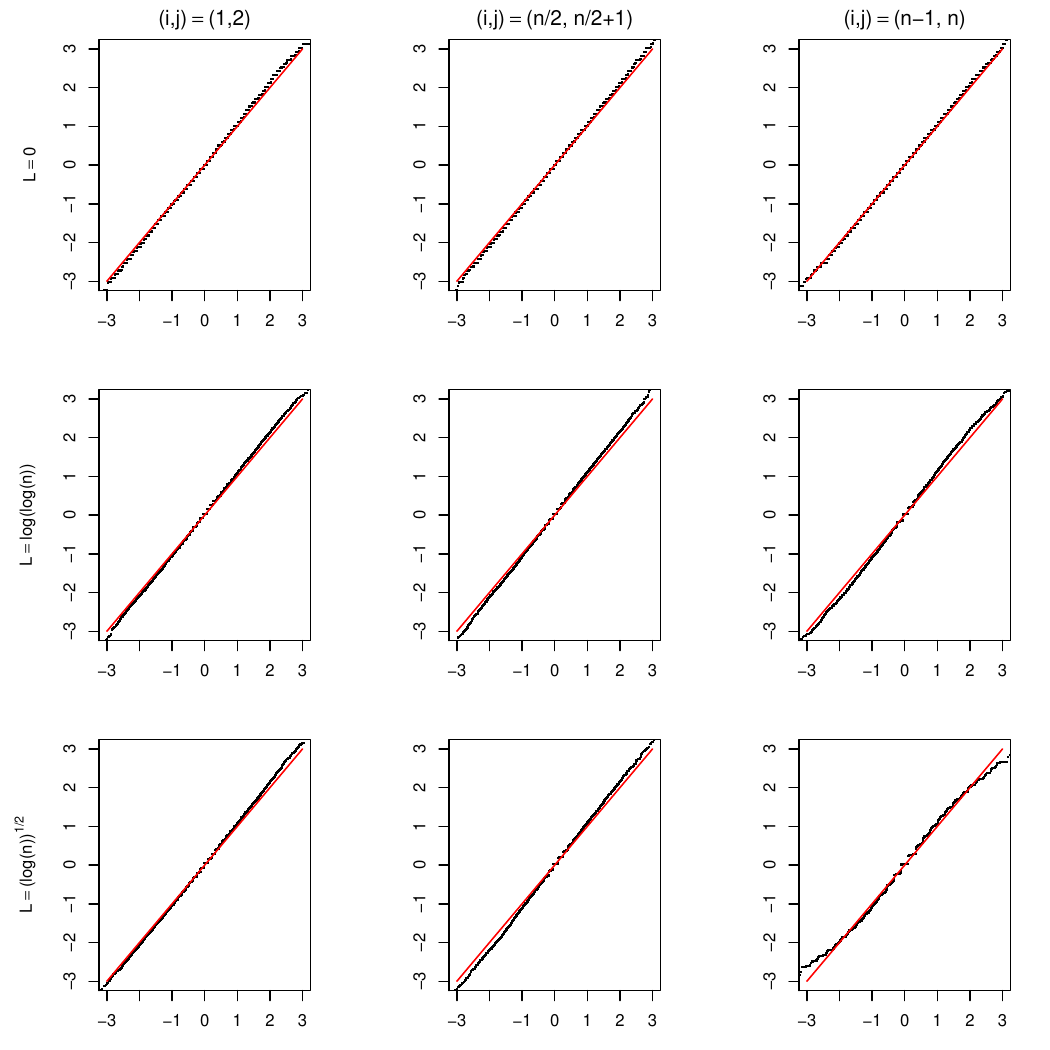}
\caption{In case of model $M_{logit}$, the QQ plots of $\xi_{ij}$ with red color for $\widehat{\xi}_{ij}$ ($n=100~~ and~~ \epsilon = 2$). }
\label{Fig2:logit}
\end{figure}

\begin{figure}[H]
\centering
\includegraphics{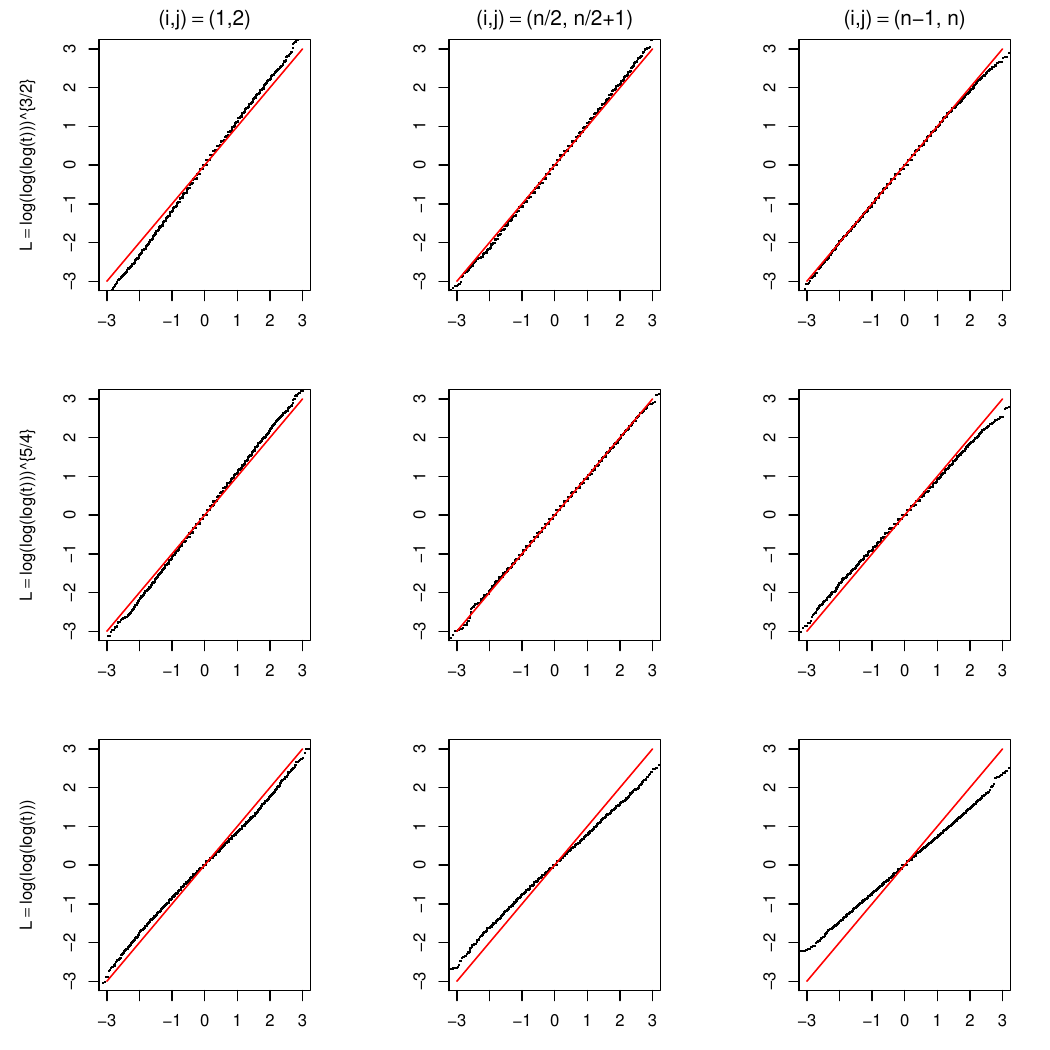}
\caption{In case of model $M_{cloglog}$, the QQ plots of $\xi_{ij}$ with red color for $\widehat{\xi}_{ij}$ ($n=100~~ and~~ \epsilon = 2$). }
\label{Fig3:cloglog}
\end{figure}

\begin{table}[htbp]
  \centering
  \caption{In case of $M_{log}$, estimated coverage probabilities of $\alpha_i-\alpha_j$ for pair $(i,j)$  as well as  the length of confidence intervals , and the probabilities that the parameter estimator does not exist, multiplied by $100$. }
\begin{tabular}{llccc}
  \hline
    \multicolumn{1}{p{4.055em}}{$n$} & \multicolumn{1}{p{4.055em}}{$(i,j)$} &$-\log(\log(t))^{1/3}$    & $-\log(\log(t))^{1/2}$    & $-\log(\log(t))$  \\
   \hline
          &       &  $a_1=0.01$, $a_2=(\Lambda-0.01)/4,m=2$   & \multicolumn{1}{r}{} & \multicolumn{1}{r}{} \\
    \hline
    100   & (1,2) & 98.64/0.41/0 & 98.69/0.43/0.22 & 96.42/0.41/3.34 \\
          & (50,51) & 96.10/0.55/0 & 96.13/0.59/0.22 & 92.59/0.60/3.34 \\
          & (99,100) & 94.14/0.73/0 & 93.67/0.80/0.22 & 91.04/0.90/3.34 \\
    \hline
    200   & (1,2) & 98.98/0.30/0 & 99.00/0.30/0 & 96.53/0.28/0.14 \\
          & (100,101) & 96.94/0.40/0 & 96.12/0.42/0 & 91.86/0.43/0.14 \\
          & (199,200) & 95.58/0.54/0 & 94.58/0.58/0 & 88.96/0.65/0.14 \\
    \hline
          &       & $a_1=\Lambda-0.01$, $a_2=0.025$,$m=2$   & \multicolumn{1}{r}{} & \multicolumn{1}{r}{} \\
    \hline
    100   & (1,2) & 98.68/0.41/0 & 98.72/0.44/0.08 & 96.52/0.42/2.42 \\
          & (50,51) & 95.82/0.55/0 & 96.32/0.59/0.08 & 92.60/0.61/2.42 \\
          & (99,100) & 94.10/0.73/0 & 93.55/0.80/0.08 & 90.65/0.90/2.42 \\
    \hline
    200   & (1,2) & 99.08/0.30/0 & 98.26/0.30/0 & 96.74/0.28/0.06 \\
          & (100,101) & 97.42/0.40/0 & 96.30/0.42/0 & 91.61/0.43/0.06 \\
          & (199,200) & 95.62/0.54/0 & 95.40/0.57/0 & 88.69/0.65/0.06 \\
    \hline
          &       & $a_1=4*\Lambda/5$, $a_2=\Lambda/5$,$m=2$   & \multicolumn{1}{r}{} & \multicolumn{1}{r}{} \\
    \hline
    100   & (1,2) & 98.60/0.42/0.4 & 98.74/0.44/0.1 & 96.53/0.42/3.16 \\
          & (50,51) & 96.26/0.55/0.04 & 96.06/0.59/0.1 & 96.27/0.61/3.16 \\
          & (99,100) & 93.70/0.73/0.04 & 94.65/0.80/0.1 & 90.27/0.90/3.16 \\
   \hline
    200   & (1,2) & 98.82/0.30/0 & 98.86/0.30/0 & 96.43/0.28/0.05 \\
          & (100,101) & 97.16/0.40/0 & 96.50/0.42/0 & 92.53/0.43/0.05 \\
          & (199,200) & 95.28/0.54/0 & 94.48/0.58/0 & 88.36/0.65/0.05 \\
    \hline
    \end{tabular}%
  \label{tab:log:1}%
\end{table}%

\begin{table}[htbp]
  \centering
  \caption{In case of $M_{logit}$, estimated coverage probabilities of $\alpha_i-\alpha_j$ for pair $(i,j)$  as well as  the length of confidence intervals , and the probabilities that the parameter estimator does not exist, multiplied by $100$. }
\begin{tabular}{llccc}
 \hline
    \multicolumn{1}{p{4.055em}}{$n$} & \multicolumn{1}{p{4.055em}}{$(i,j)$} & $0$    & $\log(log(n))$    & $\log(n)^{1/2}$ \\
    \hline
          &       & $a_1=0.01$, $a_2=(\Lambda-0.01)/4,m=2$   & \multicolumn{1}{r}{} & \multicolumn{1}{r}{} \\
    \hline
    100   & (1,2) & 93.02/0.60/0 & 92.51/0.63/1.9 & 92.07/0.68/44.24 \\
          & (50,51) & 93.16/0.60/0 & 91.76/0.76/1.9 & 91.60/0.94/44.24 \\
          & (99,100) & 92.98/0.60/0 & 90.74/1.03/1.9 & 96.10/1.63/44.24 \\
     \hline
    200   & (1,2) & 93.94/0.40/0 & 93.48/0.48/0.024 & 93.60/0.48/14.00 \\
          & (100,101) & 93.56/0.40/0 & 93.02/0.55/0.04 & 92.49/0.68/14.00 \\
          & (199,200) & 94.16/0.40/0 & 92.86/0.75/0.04 & 94.21/1.12/14.00 \\
     \hline
          &       &$a_1=\Lambda-0.01$, $a_2=0.025$,$m=2$   & \multicolumn{1}{r}{} & \multicolumn{1}{r}{} \\
     \hline
    100   & (1,2) & 93.10/0.58/0 & 92.45/0.63/1.5 & 92.50/0.68/45.1 \\
          &(50,51)& 93.00/0.58/0 & 91.53/0.76/1.5 & 91.51/0.94/45.1 \\
          & (99,100) & 93.22/0.58/0 & 91.15/1.02/1.5 & 95.70/1.60/45.1 \\
     \hline
    200   & (1,2) & 94.04/0.40/0 & 93.86/0.45/0.03 & 93.76/0.48/11.98 \\
          & (100,101) & 94.25/0.40/0 & 93.80/0.55/0.03 & 92.69/0.68/11.98 \\
          & (199,200) & 94.15/0.40/0 & 92.51/0.76/0.03 & 95.02/1.12/11.98 \\
     \hline
          &       & $a_1=4*\Lambda/5$, $a_2=\Lambda/5$,$m=2$   & \multicolumn{1}{r}{} & \multicolumn{1}{r}{} \\
     \hline
    100   & (1,2)& 92.48/0.58/0 & 92.27/0.63/1.62 & 92.46/0.68/44.58 \\
          &(50,51)& 92.64/0.58/0 & 90.93/0.76/1.62 & 90.47/0.94/44.58 \\
          & (99,100) & 92.88/0.58/0 & 90.28/1.03/1.62 & 95.23/1.59/44.58 \\
     \hline
    200   & (1,2) & 93.81/0.40/0 & 93.50/0.45/0 & 93.80/0.48/12.67 \\
          & (100,101) & 94.00/0.40/0 & 93.25/0.55/0 & 92.61/0.68/12.67 \\
          & (199,200) & 94.27/0.40/0 & 92.51/0.75/0 & 95.00/1.12/12.67 \\
     \hline
    \end{tabular}%
  \label{tab:logit:2}%
\end{table}%

\newpage
\begin{table}[htbp]
  \centering
  \caption{In case of $M_{cloglog}$, estimated coverage probabilities of $\alpha_i-\alpha_j$ for pair $(i,j)$  as well as  the length of confidence intervals , and the probabilities that the parameter estimator does not exist, multiplied by $100$. }
\begin{tabular}{llccc}
  \hline
    \multicolumn{1}{p{4.055em}}{$n$} & \multicolumn{1}{p{4.055em}}{$(i,j)$} & $log(log(log(t)))^{3/2}$    & $log(log(log(t)))^{5/4}$    & $log(log(log(t)))$ \\
   \hline
          &       & $a_1=0.01$, $a_2=(\Lambda-0.01)/4,m=2$   & \multicolumn{1}{r}{} & \multicolumn{1}{r}{} \\
    \hline
    100   & (1,2) & 86.86/0.47/0 & 89.21/0.48/0 & 91.24/0.50/0 \\
          & (50,51)& 88.16/0.45/0 & 91.52/0.47/0 & 93.50/0.50/0 \\
          & (99,100) & 89.46/0.45/0 & 92.70/0.47/0 & 96.06/0.52/0 \\
    \hline
    200   & (1,2) & 91.58/0.34/0 & 92.42/0.35/0 & 97.08/0.36/0 \\
          & (100,101) & 93.63/0.34/0 & 95.38/0.36/0 & 98.54/0.37/0 \\
          & (199,200) & 95.66/0.34/0 & 96.78/0.38/0 & 99.18/0.41/0 \\
    \hline
          &       & $a_1=\Lambda-0.01$, $a_2=0.025$,$m=2$   & \multicolumn{1}{r}{} & \multicolumn{1}{r}{} \\
    \hline
    100   &(1,2) & 86.64/0.47/0 & 89.33/0.45/0 & 91.34/0.45/0 \\
          & (50,51) & 88.06/0.48/0 & 91.46/0.47/0 & 93.52/0.47/0 \\
          & (99,100) & 89.54/0.50/0 & 92.84/0.50/0 & 95.90/0.52/0 \\
    \hline
    200   &(1,2) & 91.48/0.34/0 & 92.22/0.35/0 & 97.16/0.36/0 \\
          & (100,101) & 93.88/0.34/0 & 95.66/0.36/0 & 98.66/0.37/0 \\
          & (199,200) & 95.40/0.34/0 & 96.78/0.38/0 & 99.24/0.41/0 \\
    \hline
          &       &  $a_1=4*\Lambda/5$, $a_2=\Lambda/5$,$m=2$  & \multicolumn{1}{r}{} & \multicolumn{1}{r}{} \\
   \hline
    100   & (1,2) & 86.02/0.47/0 & 88.72/0.48/0 & 91.02/0.50/0 \\
          & (50,51) & 88.32/0.45/0 & 91.18/0.47/0 & 93.56/0.50/0 \\
          & (99,100) & 89.72/0.45/0 & 92.22/0.47/0 & 95.72/0.52/0 \\
    \hline
    200   &(1,2) & 91.42/0.34/0 & 92.06/0.35/0 & 96.70/0.36/0 \\
          & (100,101) & 93.82/0.34/0 & 95.54/0.36/0 & 98.66/0.37/0 \\
          & (199,200) & 95.18/0.34/0 & 96.94/0.38/0 & 99.12/0.41/0 \\
    \hline
    \end{tabular}%
  \label{tab:cloglog:3}%
\end{table}

\begin{table}[h!]\centering
\footnotesize
\caption{The Bruce Kapferer network dataset: In case of $M_{log}$, the parameter estimator $\widetilde{\alpha}_{log}$ in model $M_{log}$, $\widetilde{\alpha}_{log}$,$95\%$confidence intervals(in square brackets) and their standard errors(in parentheses) .}
\label{table:3}
\begin{tabular}{cccccc}
\hline
Vertex     & degree & $\widetilde{\alpha}_{log}$   & Vertex     &  degree& $\widetilde{\alpha}_{log}$    \\
\hline
  \multicolumn{6}{c}{In case of $M_{log}$ and $\epsilon =2$}\\
\hline
    1     & -2.50[-3.63,-0.76](0.73) & 2     & 20    & -0.59[-1.24,0.06](0.33) & 10 \\
    2     & -1.28[-2.19,-0.37](0.47) & 5     & 21    &-1.28[-2.19,-0.37](0.47) & 5 \\
    3     & -0.49[-1.11,0.13](0.32) & 11    & 22    & -0.49[-1.11,0.13](0.32) & 11 \\
    4     & -0.49[-1.11,0.13](0.32) & 11    & 23    & -2.50[-3.63,-0.76](0.73) & 2\\
    5     & -1.79[-2.96,-0.62](0.60) & 3     & 24    & -2.50[-3.63,-0.76](0.73) & 2\\
    6     & -1.10[-1.93,-0.27](0.42) & 6     & 25    & -1.79[-2.96,-0.62](0.60) & 3 \\
    7     & -1.10[-1.93,-0.27](0.42) & 6     & 26    & -0.69[-1.38,0.01](0.35) & 9 \\
    8     & -1.79[-2.96,-0.62](0.60) & 3     & 27    & -0.69[-1.38,0.01](0.35) & 9 \\
    9     & -0.69[-1.38,0.01](0.35) & 9     & 28    & -0.59[-1.24,0.06](0.33) & 10 \\
    10    & -2.89[-4.91,-0.87](1.03) & 1     & 29    &-1.10[-1.93,-0.27](0.42) & 6 \\
    11    & -0.18[-0.72,0.35](0.27) & 15    & 30    & -0.12[-0.64,0.40](0.26) & 16 \\
    12    & -0.25[-0.80,0.30](0.28) & 14    & 31    & -0.59[-1.24,0.06](0.33) & 10 \\
    13    & -0.59[-1.24,0.06](0.33) & 10    & 32    & -0.12[-0.64,0.40](0.26) & 16 \\
    14    & -0.81[-1.53,-0.09](0.37) & 8     & 33    & -1.10[-1.93,-0.27](0.42)& 6 \\
    15    & -1.28[-2.19,-0.37](0.47) & 5     & 34    & -0.59[-1.24,0.06](0.33) & 10\\
    16    & 0.37[-0.05,0.78](0.21) & 26    & 35    & -1.10[-1.93,-0.27](0.42) & 6 \\
    17    & -0.94[-1.72,-0.17](0.39) & 7     & 36    & -0.69[-1.38,0.01](0.35) & 9 \\
    18    & -0.06[-0.56,0.45](0.26) & 17    & 37    &-1.28[-2.19,-0.37](0.47) & 5 \\
    19    & -2.89[-4.91,-0.87](1.03) & 1     &       &       &  \\
  \hline
    \end{tabular}%
  \label{tab:example:log}%
\end{table}%

%
\begin{table}[h!]\centering
\footnotesize
\caption{The Bruce Kapferer network dataset: In case of $M_{logit}$, the parameter estimator $\widetilde{\alpha}_{logit}$ in model $M_{logit}$, $\widetilde{\alpha}_{logit}$,$95\%$confidence intervals(in square brackets) and their standard errors(in parentheses) .}
\label{table:3}
\begin{tabular}{cccccc}
\hline
Vertex     & degree & $\widetilde{\alpha}_{logit}$   & Vertex     &  degree& $\widetilde{\alpha}_{logit}$    \\
\hline
  \multicolumn{6}{c}{In case of $M_{logit}$ and $\epsilon =2$}\\
\hline
    1     & -2.52[-4.02,-1.03](0.76) & 2     & 20    & -0.32[-1.12,0.49](0.41) & 10 \\
    2     & -1.38[-2.39,-0.36](0.52) & 5     & 21    & -1.38[-2.39,-0.36](0.52) & 5 \\
    3     & -0.14[-0.93,0.64](0.40) & 11    & 22    & -0.14[-0.93,0.64](0.40) & 11 \\
    4     & -0.14[-0.93,0.64](0.40) & 11    & 23    & -2.52[-4.02,-1.03](0.76) & 2 \\
    5     & -2.04[-3.29,-0.78](0.64) & 3     & 24    & -2.52[-4.02,-1.03](0.76) & 2 \\
    6     & -1.12[-2.07,-0.17](0.48) & 6     & 25    & -2.04[-3.29,-0.78](0.64) &3\\
    7     & -1.12[-2.07,-0.17](0.48) & 6     & 26    & -0.50[-1.32,0.33](0.42) & 9 \\
    8     & -2.04[-3.29,-0.78](0.64) & 3     & 27    &-0.50[-1.32,0.33](0.42) & 9 \\
    9     & -0.50[-1.32,0.33](0.42) & 9     & 28    & -0.32[-1.12,0.49](0.41) & 10 \\
    10    & -3.30[-5.35,-1.26](1.04) & 1     & 29    & -1.12[-2.07,-0.17](0.48) & 6 \\
    11    & 0.47[-0.26,1.21](0.38) & 15    & 30    & 0.62[-0.11,1.35](0.37)&16 \\
    12    & 0.33[-0.42,1.07](0.38) & 14    & 31    & -0.32[-1.12,0.49](0.41) & 10 \\
    13    & -0.32[-1.12,0.49](0.41) & 10    & 32    & 0.62[-0.11,1.35](0.37) & 16 \\
    14    & -0.69[-1.55,0.17](0.44) & 8     & 33    & -1.12[-2.07,-0.17](0.48) & 6 \\
    15    & -1.38[-2.39,-0.36](0.52) & 5     & 34    & -0.32[-1.12,0.49](0.41) & 10 \\
    16    & 2.10[1.29,2.91](0.41) & 26    & 35    & -1.12[-2.07,-0.17](0.48) & 6 \\
    17    & -0.89[-1.79,0.00](0.46) & 7     & 36    & -0.50[-1.32,0.33](0.42) & 9 \\
    18    & 0.76[0.03,1.49](0.37) & 17    & 37    & -1.38[-2.39,-0.36](0.52) & 5 \\
    19    & -3.30[-5.35,-1.26](1.04) & 1     &       &       &  \\
    \hline
    \end{tabular}%
  \label{tab:example:logit}%
\end{table}%

\newpage	

\begin{table}[h!]\centering
\footnotesize
\caption{The Bruce Kapferer network dataset: In case of $M_{cloglog}$, the parameter estimator $\widetilde{\alpha}_{cloglog}$ in model $M_{cloglog}$, $\widetilde{\alpha}_{cloglog}$,$95\%$confidence intervals(in square brackets) and their standard errors(in parentheses) .}
\label{table:3}
\begin{tabular}{cccccc}
\hline
Vertex     & degree & $\widetilde{\alpha}_{cloglog}$   & Vertex     &  degree& $\widetilde{\alpha}_{cloglog}$    \\
\hline
  \multicolumn{6}{c}{In case of $M_{cloglog}$ and $\epsilon =2$}\\
\hline
    1     & -1.19[-2.25,-0.14](0.54) & 2     & 20    & -0.17[-0.85,0.52](0.35)& 10 \\
    2     & -0.70[-1.50,0.09](0.40) & 5     & 21    & -0.70[-1.50,0.09](0.40)& 5\\
    3     & -0.07[-0.75,0.60](0.34) & 11    & 22    & -0.07[-0.75,0.60](0.34)& 11 \\
    4     & -0.07[-0.75,0.60](0.34) & 11    & 23    & -1.19[-2.25,-0.14](0.54) & 2 \\
    5     & -0.10[-1.93,-0.07](0.47) & 3     & 24    & -1.19[-2.25,-0.14](0.54) & 2 \\
    6     & -0.58[-1.34,0.16](0.39) & 6     & 25    & -0.10[-1.93,-0.07](0.47) & 3 \\
    7     & -0.58[-1.34,0.16](0.39) & 6     & 26    & -0.26[-0.96,0.43](0.35) & 9 \\
    8     & -0.10[-1.93,-0.07](0.47) & 3     & 27    & -0.26[-0.96,0.43](0.35) & 9 \\
    9     & -0.26[-0.96,0.43](0.35) & 9     & 28    & -0.17[-0.85,0.52](0.35) & 10 \\
    10    & -1.48[-2.81,-0.15](0.68) & 1     & 29    &-0.58[-1.34,0.16](0.39) & 6 \\
    11    & 0.26[-0.40,0.93](0.34) & 15    & 30    & 0.34[-0.33,1.01](0.34) & 16 \\
    12    & 0.18[-0.48,0.85](0.34) & 14    & 31    & -0.17[-0.85,0.52](0.35) & 10 \\
    13    & -0.17[-0.85,0.52](0.35) & 10    & 32    &0.34[-0.33,1.01](0.34) & 16\\
    14    & -0.36[-1.07,0.35](0.36) & 8     & 33    & -0.58[-1.34,0.16](0.39) & 6 \\
    15    & -0.70[-1.50,0.09](0.40) & 5     & 34    & -0.17[-0.85,0.52](0.35) & 10\\
    16    & 1.05[0.31,1.79](0.38) & 26    & 35    & -0.58[-1.34,0.16](0.39) & 6 \\
    17    & -0.47[-1.20,0.26](0.37) & 7     & 36    &-0.26[-0.96,0.43](0.35) & 9 \\
    18    & 0.42[-0.25,1.09](0.34) & 17    & 37    & -0.70[-1.50,0.09](0.40) & 5 \\
    19    & -1.48[-2.81,-0.15](0.68) & 1     &       &       &  \\
\hline
    \end{tabular}%
  \label{tab:example:cloglog}%
\end{table}%

\begin{figure}[!htp]
\centering
\caption{The scatter plots ($n=39,\epsilon =2$). The $\bar{d}$ denotes the noisy degree sequences and $\widehat{\alpha}$ denotes the corresponding the parameter estimation.}
\label{Fig4}
\includegraphics[width=0.90\textwidth]{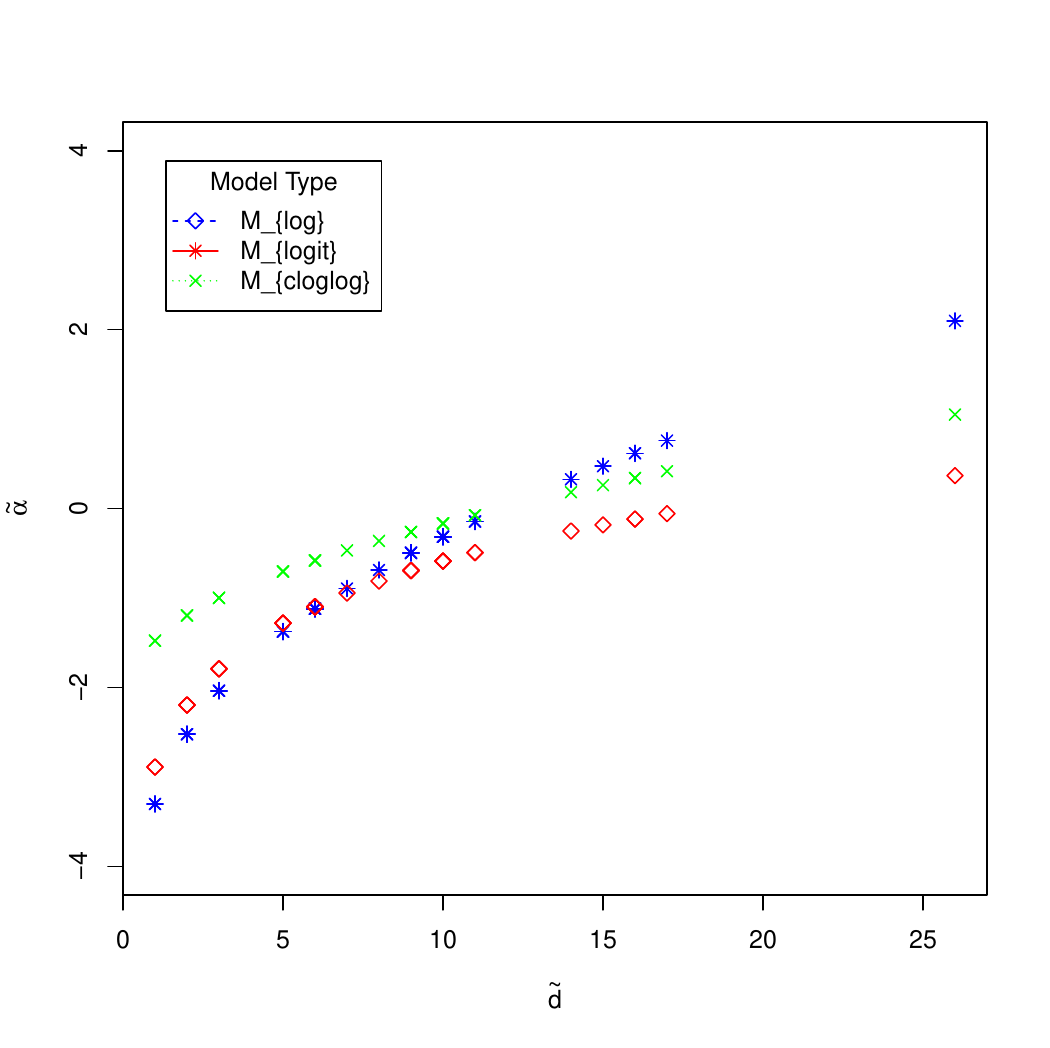}
\end{figure}

\end{document}